\documentclass[12pt]{amsart}
\usepackage[pdftex,colorlinks,citecolor=blue,urlcolor=blue]{hyperref}
\usepackage[margin=2.75cm]{geometry}
\usepackage{amssymb}
\usepackage{graphicx,color}
\usepackage{amsmath}
\usepackage{extpfeil}
\usepackage{comment}
\usepackage{bbm}
\usepackage{MnSymbol}
\usepackage{cleveref}
\usepackage{enumitem}

\newtheorem{theorem}{Theorem}[section]
\newtheorem{theoremM}{Theorem}

\newtheorem{proposition}[theorem]{Proposition}
\newtheorem{lemma}[theorem]{Lemma}
\newtheorem{corollary}[theorem]{Corollary}

\theoremstyle{definition}
\newtheorem{definition}[theorem]{Definition}
\newtheorem*{definitionM}{Definition}
\newtheorem{notation}[theorem]{Notation}
\newtheorem{convention}[theorem]{Convention}
\newtheorem{example}[theorem]{Example}

\newtheorem{question}[theorem]{Question}

\newtheorem{remark}[theorem]{Remark}

\newtheorem{setting}[theorem]{Setting}

\newcommand{\Hc}{\mathcal{H}}
\newcommand{\Ic}{\mathcal{I}}
\newcommand{\Jc}{\mathcal{J}}
\newcommand{\Kc}{\mathcal{K}}

\newcommand{\NN}{\mathbb{N}}

\newcommand{\init}{\mathop{\rm in}\nolimits}
\newcommand{\rev}{\mathop{<_{\rm rev}}\nolimits}

\newcommand{\HF}{\mathop{\rm HF}\nolimits}

\newcommand{\field}{\mathbbm{k}}
\newcommand{\pf}{\mathrm{pf}}

\newcommand{\Xgen}{X^{\mathrm{gen}}}
\newcommand{\Xsym}{X^{\mathrm{sym}}}
\newcommand{\Xsk}{X^{\mathrm{skew}}}

\newcommand{\Sgen}{S^{\mathrm{gen}}}
\newcommand{\Ssym}{S^{\mathrm{sym}}}
\newcommand{\Ssk}{S^{\mathrm{skew}}}

\renewcommand\labelenumi{(\roman{enumi})}
\renewcommand\theenumi\labelenumi

\title[On strongly Koszul algebras and tidy Gr\"obner bases]{On strongly Koszul algebras \\and tidy Gr\"obner bases}
\author{Alessio D'Al\`i}
\address{Dipartimento di Matematica, Politecnico di Milano, Italy}
\email{alessio.dali@polimi.it}

\begin{document}

\begin{abstract}
Strongly Koszul algebras were introduced by Herzog, Hibi and Restuccia in 2000. The goal of the present paper is to provide an in-depth study of such algebras and to investigate how strong Koszulness interacts with the existence of a quadratic Gr\"obner basis for the defining ideal. Firstly, we prove that the existence of a quadratic revlex-universal Gr\"obner basis with a strong sparsity condition (that we name ``tidiness'') is a sufficient condition for strong Koszulness, and exhibit several concrete examples arising from determinantal objects and Macaulay's inverse system. We then prove that there exist standard graded algebras that are strongly Koszul but do not admit a Gr\"obner basis of quadrics even after a linear change of coordinates, thus answering negatively a question posed by Conca, De Negri and Rossi. As a bonus, we prove that strong Koszulness behaves well under tensor and fiber products of algebras and illustrate how Severi varieties and Macaulay's inverse system interact to produce examples of strongly Koszul algebras with a geometric flavor.
\end{abstract}

\maketitle

\tableofcontents

\section{Introduction}

A standard graded $\field$-algebra $R$ is said to be Koszul if the minimal graded free resolution of the residue field $\field$ as an $R$-module is linear. Since their introduction by Priddy in 1970 \cite{Priddy}, Koszul algebras have been a pivotal object in commutative algebra both for their ubiquity and their useful properties; for more information, we refer the interested reader to the surveys \cite{CDR} and \cite{Fr}. Since checking Koszulness by making use of the definition only is often an arduous task, mathematicians have been interested in finding necessary and sufficient conditions for Koszulness, introducing several variations and strengthenings in the process. In the present paper we examine one of such strengthenings, namely, the concept of \emph{strong Koszulness}. The first definition of a strongly Koszul algebra was given by Herzog, Hibi and Restuccia in 2000 \cite{HHR}; throughout this paper, however, we will adopt a later definition due to Conca, De Negri and Rossi \cite{CDR}, which is a priori more restrictive but easier to use. 

\begin{definitionM}[strongly Koszul algebra {\cite[Definition 3.11]{CDR}}]
    Let $R$ be a standard graded $\field$-algebra, where $\field$ is a field. If there exists a $\field$-basis $B$ of $R_1$ such that, for any proper subset $A \subsetneq B$ and any $b \in B \setminus A$, the colon ideal $(a \mid a \in A) :_R b$ is generated by elements of $B$, we say that $R$ is \emph{strongly Koszul} (with respect to $B$).
\end{definitionM}

The Conca--De Negri--Rossi definition does not rely on an ordering of the $\field$-basis $B$, whereas the original Herzog--Hibi--Restuccia one does. In the toric case, the two definitions are known to be equivalent: see, e.g., Murai \cite[Lemma 3.18 and Lemma 3.19]{MuraiKoszul}. However, it is not known to us whether a standard graded algebra that is strongly Koszul with respect to the Herzog--Hibi--Restuccia definition but not with respect to the Conca--De Negri--Rossi one exists.\footnote{The ideal defining such an object should be generated by at least two quadrics and not admit a tidy revlex-universal Gr\"obner basis of quadrics by \Cref{introthm:main} below: in particular, such an ideal should be neither monomial nor toric, see \Cref{cor:monomial strongly Koszul} and \Cref{cor:toric strongly Koszul}.}

Strongly Koszul algebras inspired Conca, Trung and Valla to introduce Koszul filtrations \cite{ConcaTrungValla} as a more flexible tool to establish Koszulness for a given algebra. Indeed, strongly Koszul algebras form a rather small subset of all Koszul algebras; moreover, it is not always clear which $\field$-basis of $R_1$ one should pick as a candidate. Because of this latter hindrance, mathematicians have often studied strong Koszulness in the presence of some additional structure: a popular choice is asking for the algebra to be toric (see for instance \cite{MatsudaOhsugi}, \cite[Section 3.3]{MuraiKoszul}, \cite[Section 6]{2xe}, \cite{HibiMatsudaOhsugi}, \cite{OhsugiHerzogHibi}), while Nguyen addressed the toric face ring case in \cite{NguyenToricFaceRing}. More recently, VandeBogert has studied in \cite{vandebogert} the behavior of certain monomial ideals living in strongly Koszul algebras.

\noindent Strong Koszulness has made its appearance also outside of commutative algebra: indeed, noncommutative analogues of strongly Koszul algebras are featured for instance in \cite{piontkovskii}, \cite{noncommutative1} and \cite{noncommutative2}.

\subsection{What's in this paper?}
The aim of this paper is to get a better understanding of strong Koszulness, with a focus on its interactions with Gr\"obner-related conditions. Our first result shows that strong Koszulness behaves well with respect to tensor and fiber products of standard graded $\field$-algebras, in the same vein as what was proved by Herzog, Hibi and Restuccia in the toric case \cite[Proposition 2.3(b)]{HHR}. In the statement below, $\iota_i$ is the natural injection of $R_{(i)}$ inside the relevant tensor or fiber product: see \Cref{setting} and \Cref{notation:tensor and fiber}.

\begin{theoremM}[see \Cref{thm:strongly Koszul tensor and fiber products}] \label{introthm:tensor and fiber}
A tensor or fiber product of two standard graded $\field$-algebras $R_{(1)}$ and $R_{(2)}$ is strongly Koszul with respect to the $\field$-basis $\iota_1(B_1) \cup \iota_2(B_2)$ if and only if each $R_{(i)}$ is strongly Koszul with respect to $B_i$.
\end{theoremM}

We then turn to finding sufficient conditions for strong Koszulness that are easier to check than the definition itself. To state our main result, a new definition is in order:

\begin{definitionM}[tidy polynomial; see \Cref{def:tidy}]
Let $S = \field[x_1, \ldots, x_n]$ and let $f$ be a homogeneous polynomial in $S$. If each variable of $S$ divides at most one monomial in the support of $f$, we say that $f$ is \emph{tidy}. If a set $\mathcal{F} \subseteq S$ consists of tidy polynomials, we say that $\mathcal{F}$ is tidy.
\end{definitionM}

Examples of tidy polynomials include monomials and generators of homogeneous toric ideals: see \Cref{ex:tidy examples} for more instances. When the homogeneous ideal $I \subseteq S$ admits a degree reverse lexicographic Gr\"obner basis made of tidy polynomials, certain colon ideals over $S/I$ become monomial (see \Cref{prop:monomial colon ideals} for a detailed statement). Building on this observation, we get the core technical result of the present paper:

\begin{theoremM}[see \Cref{thm:main theorem}] \label{introthm:main}
If the homogeneous ideal $I \subseteq S =\field[x_1, \ldots, x_n]$ can be generated by a set of tidy quadrics forming a Gr\"obner basis with respect to all possible degree reverse lexicographic orders, then the standard graded $\field$-algebra $S/I$ is strongly Koszul with respect to the variables.
\end{theoremM}

\Cref{introthm:main} is our main tool to certify strong Koszulness. As a first application, we show that any quotient of a polynomial ring by a single homogeneous quadric is strongly Koszul (\Cref{prop:quadric hypersurface}). We also show that the ideal of $2$-minors of either a generic or a generic symmetric matrix defines a strongly Koszul algebra, and the same holds if we make the matrix sparser by introducing any number of zeros (\Cref{thm:2-minors}).

Extending previous work by Shafiei \cite{Sha}, we then turn to investigating the interaction between Macaulay's inverse system and certain determinantal objects. We prove that the ideal apolar to the module generated by maximal minors of a generic matrix defines a strongly Koszul algebra (\Cref{prop:generic apolarity}), and so does the ideal apolar to the module generated by maximal Pfaffians of a generic skew-symmetric matrix (\Cref{cor:Pfaffian apolarity}). This analysis prompts a further intriguing question with a geometric flavor: let $\field=\mathbb{C}$ and $X$ be one of the four \emph{Severi varieties} (see \Cref{subsec:Severi} for a detailed explanation). The secant variety of lines to any Severi variety is a cubic hypersurface of determinantal nature. Does the apolar ideal to this cubic polynomial define a strongly Koszul algebra? The answer turns out to be positive:

\begin{theoremM}[see \Cref{thm:Severi}] \label{introthm:Severi} 
    Let $X \subseteq \mathbb{P}^N_{\mathbb{C}}$ be a Severi variety, let $S$ be the coordinate ring of $\mathbb{P}^N$ and let $F$ be the cubic polynomial defining the secant variety of lines of $X$. Then the Artinian Gorenstein ring $S/\mathrm{ann}(F)$ is strongly Koszul with respect to the variables.
\end{theoremM}

The proof of \Cref{introthm:Severi} requires quite a lot of combinatorial gymnastics, especially for the case when $X$ is the \emph{Cayley plane}, where the intersection patterns of the 27 lines on a smooth cubic enter the picture. It is currently unclear to us how to use geometry to shed further light on \Cref{introthm:Severi} or to simplify its proof, and we leave these questions open for future work.

Finally, we answer negatively a question by Conca, De Negri and Rossi \cite[Question 4.13(1)]{CDR} about the relation between strong Koszulness and the existence of a quadratic Gr\"obner basis for the defining ideal.

\begin{theoremM}[see \Cref{prop:smallest counterexample} and \Cref{prop:many counterexamples}] \label{introthm:strongly Koszul but no qGB}
There exist strongly Koszul algebras whose defining ideal does not admit a Gr\"obner basis of quadrics even after a linear change of coordinates.
\end{theoremM}

\subsection{Structure of the paper}
In \Cref{sec:strong Koszulness} we examine some direct consequences of the Conca--De Negri--Rossi definition of strong Koszulness and prove \Cref{introthm:tensor and fiber}. In \Cref{sec:tidy} we introduce tidy polynomials and prove \Cref{introthm:main}. \Cref{sec:applications} contains some applications of \Cref{introthm:main}: in particular, one has that a quadric hypersurface is always strongly Koszul (\Cref{prop:quadric hypersurface}) and that coordinate sections of $S/I$, where $I$ is defined by $2$-minors of a generic or generic symmetric matrix, are strongly Koszul algebras with respect to the variables (\Cref{thm:2-minors}). In \Cref{sec:apolarity and strong Koszulness} we investigate classes of strongly Koszul algebras arising from the interaction between Macaulay's inverse system and certain determinantal objects (see \Cref{prop:generic apolarity} and \Cref{cor:Pfaffian apolarity}), proving in particular \Cref{introthm:Severi}. Finally, in \Cref{sec:not G-quadratic} we prove \Cref{introthm:strongly Koszul but no qGB}, i.e., we construct examples of strongly Koszul algebras whose defining ideal does not admit a Gr\"obner basis of quadrics even after a linear change of coordinates (\Cref{prop:smallest counterexample} and \Cref{prop:many counterexamples}).

\phantom{ }

\noindent \textbf{Acknowledgements.} The author thanks Mats Boij, Aldo Conca, Ritvik Ramkumar, and Keller VandeBogert for their valuable comments. Special thanks go to Alessio Sammartano for many useful discussions on the topics of the present paper. The author is a member of INdAM--GNSAGA and has been partially supported by the PRIN 2020355B8Y grant ``Squarefree Gr\"obner degenerations, special varieties and related topics'' and by the PRIN 2022K48YYP grant ``Unirationality, Hilbert schemes, and singularities''.

\section{Strong Koszulness under algebra operations} \label{sec:strong Koszulness}

\begin{lemma} \label{lem:SK and quotients}
Let $R$ be a standard graded $\field$-algebra and let $B$ be a $\field$-basis of $R_1$. Let $A \subseteq B$ and $R' := R/(a \mid a \in A)$. If $R$ is strongly Koszul with respect to $B$, then $R'$ is strongly Koszul with respect to $\{\overline{b} \mid b \in B \setminus A\}$.
\end{lemma}
\begin{proof}

    For this proof, let $\pi\colon R \twoheadrightarrow R'$ denote the natural surjection. 
    If $A=B$, then $R' = \field$ and the claim vacuously holds. Let us now assume that $A \subsetneq B$. Let $C \subsetneq B \setminus A$ and $b \in B \setminus (A \cup C)$, and let $\mathcal{A}$ and $\mathcal{C}$ be the ideals of $R$ generated respectively by the elements in $A$ and in $C$. We need to prove that the colon ideal $(\overline{c} \mid c \in C) :_{R'} \overline{b} = \pi(\mathcal{C}) :_{R'} \overline{b}$ is generated by nonzero elements of $\pi(B)$. One has that \[\pi(\mathcal{C}) :_{R'} \overline{b} = \pi({(\mathcal{C} + \mathcal{A}) :_R b}) = \pi((f \mid f \in C \cup A) :_R b).\]
    Since by assumption $R$ is strongly Koszul with respect to $B$, the colon ideal $(f \mid f \in C \cup A) :_R b$ is generated by elements of $B$, and hence $\pi(\mathcal{C}) :_{R'} \overline{b}$ is generated by elements of $\pi(B)$, which ends the proof.
\end{proof}

It is our next goal to investigate the behavior of strong Koszulness under tensor and fiber products of standard graded $\field$-algebras. Let us first recall some general facts and establish notation for later use.
 
\begin{setting} \label{setting}
    Let $R_{(1)}$ and ${R_{(2)}}$ be standard graded $\field$-algebras and let $\mathfrak{m}_{(1)}$ and $\mathfrak{m}_{(2)}$ be their homogeneous maximal ideals.
\begin{itemize}
    \item Consider the ring $R_{(1)} \otimes_{\field} R_{(2)}$, the tensor product of $R_{(1)}$ and $R_{(2)}$ over $\field$. Since $R_{(1)}$ is flat as a $\field$-vector space, the short exact sequence of $\field$-vector spaces $0 \to \field \to R_{(2)} \to \mathfrak{m}_{(2)} \to 0$ gives rise to an injection $\iota_1 \colon R_1 \to R_1 \otimes_{\field} R_2$ sending $r_1$ to $r_1 \otimes 1$. One gets an injection $\iota_2 \colon R_2 \to R_1 \otimes_{\field} R_2$ analogously. The tensor product $R_{(1)} \otimes_{\field} R_{(2)}$ is a standard graded $\field$-algebra with homogeneous maximal ideal $(\mathfrak{m}_{(1)})^{\mathrm e} + (\mathfrak{m}_{(2)})^{\mathrm e},$ where $(\mathfrak{m}_{(i)})^{\mathrm{e}}$ is the extension of the $R_{(i)}$-ideal $\mathfrak{m}_{(i)}$ with respect to $\iota_i$. Note that $(\mathfrak{m}_{(1)})^{\mathrm e} = \mathfrak{m}_{(1)} \otimes_{\field} R_{(2)}$ and $(\mathfrak{m}_{(2)})^{\mathrm e} = R_{(1)} \otimes_{\field} \mathfrak{m}_{(2)}$. The graded structure of $R_{(1)} \otimes_{\field} R_{(2)}$ is given by $(R_{(1)} \otimes_{\field} R_{(2)})_i = \bigoplus_{j=0}^{i}(R_{(1)})_j \otimes_{\field} (R_{(2)})_{i-j}$. In particular, if $B_i$ is a $\field$-basis of $(R_{(i)})_1$, then $\iota_1(B_1) \cup \iota_2(B_2)$ is a $\field$-basis of $(R_{(1)} \otimes_{\field} R_{(2)})_1$.
    \item The \emph{fiber product} of $R_{(1)}$ and ${R_{(2)}}$ (over $\field$) is
    \[R_{(1)} \circ {R_{(2)}} := \{(r_1, r_2) \in R_{(1)} \times {R_{(2)}} \mid \pi_{1}(r_1) = \pi_{2}(r_2)\},\]
    where each $\pi_{i} \colon {R_{(i)}} \twoheadrightarrow \field$ is the canonical surjection; in other words, $R_{(1)} \circ {R_{(2)}}$ is the pullback of the diagram $R_{(1)} \xtwoheadrightarrow{\pi_{1}} \field \xtwoheadleftarrow{\pi_{2}} {R_{(2)}}$. If $f \in R_{(i)}$, write $f$ as $\lambda + f'$, where $\lambda \in \field$ and $f' \in \mathfrak{m}_{(i)}$; this can be done in one and only one way, since $R_{(i)} \cong \field \oplus \mathfrak{m}_{(i)}$ as $\field$-vector spaces. The map $\iota_i\colon R_{(i)} \to R$ sending $f$ to $(f, \lambda)$ then yields an injection of $R_{(i)}$ into $R$. 
    The fiber product $R_{(1)} \circ R_{(2)}$ is a standard graded $\field$-algebra with maximal homogeneous ideal $(\mathfrak{m}_{(1)})^{\mathrm{e}} + (\mathfrak{m}_{(2)})^{\mathrm{e}},$ where $(\mathfrak{m}_{(i)})^{\mathrm{e}}$ is the extension of the $R_{(i)}$-ideal $\mathfrak{m}_{(i)}$ with respect to $\iota_i$. Homogeneous elements of $R_{(1)} \circ R_{(2)}$ of degree $i$ are precisely those pairs $(f,g)$ with $\deg_{R_{(1)}}(f) = \deg_{R_{(2)}}(g) = i$ (including the case when one or both of $f$ and $g$ are zero). If $B_i$ is a $\field$-basis of $(R_{(i)})_1$, then $\iota_1(B_1) \cup \iota_2(B_2)$ is a $\field$-basis of $(R_{(1)} \circ R_{(2)})_1$.
    
    \noindent Even though we will not need it here, we also note that the fiber product $R_{(1)} \circ {R_{(2)}}$ can be presented as
    \[R_{(1)} \circ {R_{(2)}} \cong \frac{R_{(1)} \otimes_{\field} {R_{(2)}}}{{\mathfrak{m}_{(1)}} \otimes_{\field} {\mathfrak{m}_{(2)}}},\]
    and the isomorphism preserves the standard graded structures described above.
\end{itemize}
\end{setting}

\begin{notation} \label{notation:tensor and fiber}
     Let $R_{(1)}$ and $R_{(2)}$ be standard graded $\field$-algebras and let $R$ be either the tensor product $R_{(1)} \otimes_{\field} R_{(2)}$ or the fiber product $R_{(1)} \circ R_{(2)}$. We will denote by $\iota_i$ the natural injection $R_{(i)} \to R$. If $J_{(i)}$ is an ideal of $R_{(i)}$, we will denote by $(J_{(i)})^{\mathrm{e}}$ the extension of the ideal $J_{(i)}$ with respect to $\iota_i$, i.e., the ideal of $R$ generated by the images of the generators of $J_{(i)}$ via $\iota_i$.
\end{notation}

The next technical statements about tensor products will prove useful in the rest of the section.

\begin{lemma} \label{lem:vanishing Tor_1}
Let $R_{(1)}$ and ${R_{(2)}}$ be standard graded $\field$-algebras, ${J_{(1)}} \subseteq R_{(1)}$ and ${J_{(2)}} \subseteq {R_{(2)}}$ be ideals, and $R = R_{(1)} \otimes_{\field} {R_{(2)}}$. Then \[\mathrm{Tor}^{R}_1\left(\frac{R}{{J_{(1)}} \otimes_{\field} {R_{(2)}}}, \ \frac{R}{R_{(1)} \otimes_{\field} {J_{(2)}}}\right) = 0.\]
\end{lemma}
\begin{proof}
Note first that (as $R_{(1)}$ and ${R_{(2)}}$ are flat as $\field$-modules, since $\field$ is a field) \[\frac{R}{{J_{(1)}} \otimes_{\field} {R_{(2)}}} \cong \frac{R_{(1)}}{{J_{(1)}}} \otimes_{\field} {R_{(2)}}\text{ and } \frac{R}{R_{(1)} \otimes_{\field} {J_{(2)}}} \cong R_{(1)} \otimes_{\field} \frac{{R_{(2)}}}{{J_{(2)}}}.\] It then follows from \cite[\S 2.3]{Jorgensen} that
\[
\begin{split}
    &\mathrm{Tor}^{R}_1\left(\frac{R}{{J_{(1)}} \otimes_{\field} {R_{(2)}}}, \ \frac{R}{R_{(1)} \otimes_{\field} {J_{(2)}}}\right)\\ \cong &\left(\mathrm{Tor}^{R_{(1)}}_0\left(\frac{R_{(1)}}{{J_{(1)}}}, R_{(1)}\right) \otimes_{\field} \mathrm{Tor}^{{R_{(2)}}}_1\left({R_{(2)}}, \frac{{R_{(2)}}}{{J_{(2)}}}\right)\right) \oplus \left(\mathrm{Tor}^{R_{(1)}}_1\left(\frac{R_{(1)}}{{J_{(1)}}}, R_{(1)}\right) \otimes_{\field} \mathrm{Tor}^{{R_{(2)}}}_0\left({R_{(2)}}, \frac{{R_{(2)}}}{{J_{(2)}}}\right)\right) = 0.
\end{split}
\]
\end{proof}

\begin{lemma} \label{lem:colon equality}
Let $R_{(1)}$ and ${R_{(2)}}$ be standard graded $\field$-algebras, ${J_{(1)}} \subseteq R_{(1)}$ and ${J_{(2)}} \subseteq {R_{(2)}}$ be ideals and $f \in R_{(1)}$. Let $R = R_{(1)} \otimes_{\field}{R_{(2)}}$ and $J = {J_{(1)}} \otimes_{\field} {R_{(2)}} + R_{(1)} \otimes_{\field} {J_{(2)}}$. One has that
\[J :_R (f \otimes 1) = ({J_{(1)}} :_{R_{(1)}} f) \otimes {R_{(2)}} + R_{(1)} \otimes {J_{(2)}}.\]
\end{lemma}

\begin{proof}
Note first that $J :_R (f \otimes 1) \supseteq ({J_{(1)}} :_{R_{(1)}} f) \otimes {R_{(2)}} + R_{(1)} \otimes {J_{(2)}}$, and thus there is a canonical surjection
\begin{equation} \label{eq:surjection for Five Lemma}
    \frac{R}{({J_{(1)}} :_{R_{(1)}} f) \otimes {R_{(2)}} + R_{(1)} \otimes {J_{(2)}}} \twoheadrightarrow \frac{R}{J :_R (f \otimes 1)}.
\end{equation}
Consider the exact sequence of $R_{(1)}$-modules
\[0 \to \frac{R_{(1)}}{{J_{(1)}} \colon_{R_{(1)}} f} \xrightarrow{\cdot f} \frac{R_{(1)}}{{J_{(1)}}} \to \frac{R_{(1)}}{{J_{(1)}}+(f)}\to 0.\]
Tensoring over $\field$ with the (flat) $\field$-vector space ${R_{(2)}}$ we get the following short exact sequence of $R$-modules:
\begin{equation} \label{eq:tensored with R_2}
0 \to \frac{R}{({J_{(1)}} \colon_{R_{(1)}} f) \otimes_{\field} {R_{(2)}}} \xrightarrow{\cdot f \otimes 1} \frac{R}{{J_{(1)}} \otimes {R_{(2)}}} \to \frac{R}{{J_{(1)}} \otimes {R_{(2)}} + (f) \otimes R_{(2)}}\to 0,
\end{equation}
where we used that $(J_{(1)} + (f)) \otimes R_{(2)} = J_{(1)} \otimes R_{(2)} + (f) \otimes R_{(2)}$. 
By \Cref{lem:vanishing Tor_1} we have that 
\[
\mathrm{Tor}^{R}_1\left(\frac{R}{({J_{(1)}}+(f)) \otimes {R_{(2)}}},\ \frac{R}{R_{(1)} \otimes {J_{(2)}}}\right)=0
\]
and thus, tensoring \eqref{eq:tensored with R_2} with the $R$-module $R/(R_{(1)} \otimes {J_{(2)}})$, we obtain the short exact sequence 
\[0 \to \frac{R}{({J_{(1)}} \colon_{R_{(1)}} f) \otimes_{\field} {R_{(2)}} + R_{(1)} \otimes {J_{(2)}}} \xrightarrow{\cdot f \otimes 1} \frac{R}{J} \to \frac{R}{J + ((f) \otimes {R_{(2)}})}\to 0.\]
Comparing with the exact sequence \[0 \to \frac{R}{J:_R(f \otimes 1)} \xrightarrow{\cdot f \otimes 1} \frac{R}{J} \to \frac{R}{J + ((f) \otimes {R_{(2)}})} \to 0\]
yields the desired result (recalling the surjection \eqref{eq:surjection for Five Lemma} and applying the Five Lemma).
\end{proof}

\begin{lemma} \label{lem:fiber product lemma}
    Let $R_{(1)}$ and ${R_{(2)}}$ be standard graded $\field$-algebras, ${J_{(1)}} \subseteq \mathfrak{m}_{(1)}$ and ${J_{(2)}} \subseteq {\mathfrak{m}_{(2)}}$ be proper homogeneous ideals, and $f \in \mathfrak{m}_{(1)}$. Moreover, let $R = R_{(1)} \circ {R_{(2)}}$ be the fiber product of $R_{(1)}$ and $R_{(2)}$ over $\field$. One has that
\[(J_{(1)}^{\mathrm e} + J_{(2)}^{\mathrm e}) :_R \iota_1(f) = (J_{(1)} :_{R_{(1)}} f)^{\mathrm e} + (\mathfrak{m}_{(2)})^{\mathrm e}.\]
\end{lemma}
\begin{proof}
        Since $f \in \mathfrak{m}_{(1)}$, one has that $\iota_1(f) = (f,0)$. Analogously, $\iota_1(p) = (p,0)$ for every $p \in J_{(1)}$ and $\iota_2(q) = (0,q)$ for every $q \in J_{(2)}$, since both $J_{(1)}$ and $J_{(2)}$ are proper homogeneous ideals. In particular, one has that $p \in J_{(1)}$ if and only if $(p,0) \in J_{(1)}^{\mathrm e}$. Let us first prove that \[(J_{(1)}^{\mathrm e} + J_{(2)}^{\mathrm e}) :_R (f,0) \supseteq (J_{(1)} :_{R_{(1)}} f)^{\mathrm e} + (\mathfrak{m}_{(2)})^{\mathrm e}.\]
        
        \noindent If ${\mathfrak{m}_{(2)}} = (y_1, \ldots, y_m)$, then $(\mathfrak{m}_{(2)})^{\mathrm e} = ((0,y_1),\ldots, (0,y_m))$; since $(f,0) \cdot (0,y_j) = (0,0)$, one has that $(0,y_j) \in (0,0) :_R (f,0) \subseteq (J_{(1)}^{\mathrm e} + J_{(2)}^{\mathrm e}) :_R (f,0)$.
        
        \noindent If $J_{(1)} :_R f = (h_1, \ldots, h_s)$ then, after writing $h_i = \lambda_i + h'_i$ with $\lambda_i \in \field$ and $h'_i \in \mathfrak{m}_{(1)}$, one has that $(J_{(1)} :_R f)^{\mathrm e} = ((h_1, \lambda_1), \ldots, (h_s, \lambda_s))$. Since $(h_i, \lambda_i) \cdot (f,0) = (h_if,0)$ and $h_if \in J_{(1)}$, one has that $(h_if,0) \in J_{(1)}^{\mathrm e}$ and thus $(h_i, \lambda_i) \in J_{(1)}^{\mathrm e} :_R (f,0)$.
    
        Let us now show that \[(J_{(1)}^{\mathrm e} + J_{(2)}^{\mathrm e}) :_R (f,0) \subseteq (J_{(1)} :_{R_{(1)}} f)^{\mathrm e} + (\mathfrak{m}_{(2)})^{\mathrm e}.\]
        
        \noindent Let $(h,g) \in (J_{(1)}^{\mathrm e} + J_{(2)}^{\mathrm e}) :_R (f,0)$, and write $h = \lambda + h'$, $g = \lambda + g'$ with $\lambda \in \field$, $h' \in \mathfrak{m}_{(1)}, g' \in {\mathfrak{m}_{(2)}}$. Then $(h,g) = (h,\lambda) + (0,g')$ and, since $(0,g') \in (\mathfrak{m}_{(2)})^{\mathrm e} \cap ((0,0) :_R (f,0))$, it is now enough to prove that $(h,\lambda) \in (J_{(1)} :_{R_{(1)}} f)^{\mathrm e}$. Since $(h,\lambda) \in (J_{(1)}^{\mathrm e} + J_{(2)}^{\mathrm e}) :_R (f,0)$, one has that $(hf,0) = (h,\lambda)\cdot(f,0) \in J_{(1)}^{\mathrm e} + J_{(2)}^{\mathrm e}$, but this implies that $(hf,0) \in J_{(1)}^{\mathrm e}$ and thus that $hf \in J_{(1)}$, i.e., $h \in J_{(1)} :_{R_{(1)}} f$. 
\end{proof}

We are now able to discuss the behavior of strong Koszulness under taking tensor or fiber products. The following result was proved by Herzog, Hibi and Restuccia under the additional hypothesis that the rings in play are homogeneous toric rings, see \cite[Proposition 2.3(b)]{HHR}:

\begin{theorem} \label{thm:strongly Koszul tensor and fiber products}
    Let $R_{(1)}$ and ${R_{(2)}}$ be standard graded $\field$-algebras and let $R$ be either the tensor product $R_{(1)} \otimes_{\field} {R_{(2)}}$ or the fiber product $R_{(1)} \circ {R_{(2)}}$. Let $B_1$ be a $\field$-basis of $(R_{(1)})_1$, $B_2$ a $\field$-basis of $(R_{(2)})_1$, and $B$ the $\field$-basis of $R_1$ defined via $B = \iota_1(B_1) \cup \iota_2(B_2)$. Then $R$ is strongly Koszul with respect to $B$ if and only if each $R_{(i)}$ is strongly Koszul with respect to $B_i$.
\end{theorem}
\begin{proof}
Let us first note that the ``only if'' part of the claim is a direct consequence of \Cref{lem:SK and quotients}, since \[R_{(1)} \cong \frac{R}{(\mathfrak{m}_{(2)})^{\mathrm{e}}} \text{ and } {R_{(2)}} \cong \frac{R}{(\mathfrak{m}_{(1)})^{\mathrm{e}}}.\] 
Let us now turn to the ``if'' part.
\begin{itemize}
    \item Let $R = R_{(1)} \otimes_{\field} {R_{(2)}}$, where each ${R_{(i)}}$ is strongly Koszul with respect to $B_i$. We need to prove that $R$ is strongly Koszul with respect to \[B = \{g \otimes 1 \mid g \in B_1\} \cup \{1 \otimes h \mid h \in B_2\}.\]
    
    \noindent Let $A_1 \subsetneq B_1$, $A_2 \subseteq B_2$ and $f \in B_1 \setminus A_1$. Moreover, for every $i \in \{1,2\}$, let $J_{(i)}$ be the ${R_{(i)}}$-ideal generated by $A_i$ and let $J$ be the $R$-ideal obtained as ${J_{(1)}} \otimes_{\field} {R_{(2)}} + R_{(1)} \otimes_{\field} {J_{(2)}}$ (in other words, $J = (J_{(1)})^{\mathrm e} + (J_{(2)})^{\mathrm e}$). Without loss of generality, to prove the claim it is enough to show that $J :_R (f \otimes 1)$ is generated by elements of $B$. By \Cref{lem:colon equality}, one has that \[J :_R (f \otimes 1) = ({J_{(1)}} :_{R_{(1)}} f) \otimes {R_{(2)}} + R_{(1)} \otimes {J_{(2)}}.\] Since $R_{(1)}$ is strongly Koszul with respect to $B_1$, one has that ${J_{(1)}} :_{R_{(1)}} f$ is generated by elements of $B_1$; hence, $J :_R (f \otimes 1)$ is generated by elements of $B$, as claimed.
    
    \item Let $R = R_{(1)} \circ {R_{(2)}}$, where each ${R_{(i)}}$ is strongly Koszul with respect to $B_i$. 
    We want to prove that $R$ is strongly Koszul with respect to \[B = \{(g,0) \mid g \in B_1\} \cup \{(0,h) \mid h \in B_2\}.\]
    
    \noindent Let $A_1 \subsetneq B_1$, $A_2 \subseteq B_2$ and $b \in B_1 \setminus A_1$. Moreover, for every $i \in \{1,2\}$, let $\mathcal{A}_i$ be the proper homogeneous ${R_{(i)}}$-ideal generated by $A_i$. Without loss of generality, to prove the claim it is enough to show that $(\mathcal{A}_1^{\mathrm e} + \mathcal{A}_2^{\mathrm e}) :_R (b,0)$ is generated by elements of $B$. By \Cref{lem:fiber product lemma} it holds that \[(\mathcal{A}_1^{\mathrm e} + \mathcal{A}_2^{\mathrm e}) :_R (b,0) = (\mathcal{A}_1 :_{R_{(1)}} b)^{\mathrm e} + (\mathfrak{m}_{(2)})^{\mathrm e},\] from which the desired result follows (since $R_{(1)}$ is strongly Koszul with respect to $B_1$).
\end{itemize}
\end{proof}

\section{Tidy revlex-universal Gr\"obner bases of quadrics} \label{sec:tidy}
The aim of this section is to find a sufficient condition for strong Koszulness (\Cref{thm:main theorem}). We begin by introducing some notation that we will use throughout the paper to improve readability.

\begin{notation}
Let $S = \field[x_1, \ldots, x_n]$ be a polynomial ring and let $X = \{x_1, \ldots, x_n\}$ be the set of variables of $S$. If $Y$ is a subset of $X$, we will use $Y$ to denote also the ideal of $S$ generated by such variables (when no risk of confusing the ideal with the set arises). Moreover, if $I$ is a homogeneous ideal of $S$, we will sometimes write $\overline{Y}$ to denote the ideal $(\overline{y} \mid y \in Y)$ in the standard graded algebra $R = S/I$. 
\end{notation}

Next, we recall some useful properties of the degree reverse lexicographic (from now on, \emph{revlex}) order.

\begin{remark} \label{lem:revlex and projection}
    Let $I \subseteq S = \field[x_1, \ldots, x_n]$ be a homogeneous ideal and let $\rev$ be the revlex order on $S$ such that $x_1 > x_2 > \ldots > x_n$. Let $i \in \{1, \ldots, n\}$, let $S' = \field[x_1, \ldots, x_i]$, let $\pi_i$ be the natural projection $S \twoheadrightarrow S'$ and let $<'_{\mathrm{rev}}$ be the revlex order on $S'$ such that $x_1 > x_2 > \ldots > x_i$. Then:
    \begin{enumerate}
        \item for every nonzero homogeneous polynomial $f \in S$, one has that $f \in \ker(\pi_i)$ if and only if $\init_{<_{\mathrm{rev}}}(f) \in \ker(\pi_i)$; moreover, when $f \notin \ker(\pi_i)$, one has that $\pi_i(\init_{<_{\mathrm{rev}}}(f)) = \init_{<'_{\mathrm{rev}}}(\pi_i(f))$.
        \item if $I$ is a homogeneous ideal of $S$, one has that $\pi_i(\init_{<_{\mathrm{rev}}}(I)) = \init_{<'_{\mathrm{rev}}}(\pi_i(I))$.

    \end{enumerate}
\end{remark}

\begin{lemma}[{\cite[Lemma 4.3.7]{MonomialIdeals}}] \label{lem:revlex}
    Let $I \subseteq S = \field[x_1, \ldots, x_n]$ be a homogeneous ideal and let $\rev$ be the revlex order on $S$ such that $x_1 > x_2 > \ldots > x_n$. Then, for every $i \in [n]$, one has that:
    \begin{enumerate}
    \item $\init_{\rev}(I + (x_i, \ldots, x_n)) = \init_{\rev}(I) + (x_i, \ldots, x_n)$;
    \item $\init_{\rev}((I + (x_{i+1}, \ldots, x_n)) :_S x_i) = (\init_{\rev}(I) + (x_{i+1}, \ldots, x_n)) :_S x_i$.
    \end{enumerate}
\end{lemma}

\begin{proposition}[{\cite[Theorem 1.1]{EHH}}] \label{prop:ehh}
    Let $I \subseteq S = \field[x_1, \ldots, x_n]$ be a homogeneous ideal, let $R = S/I$ and let $\rev$ be the revlex order on $S$ such that $x_1 > x_2 > \ldots > x_n$. Assume that $I$ admits a Gr\"obner basis of quadrics with respect to $\rev$. Then, for every $i \in [n]$, the colon ideal $(I + (x_{i+1}, x_{i+2}, \ldots, x_n)) :_S x_i$ is generated by linear forms modulo $I$. Thus, the colon ideal $(\overline{x}_{i+1}, \overline{x}_{i+2}, \ldots, \overline{x}_n) :_R \overline{x}_i$ is generated by linear forms.
\end{proposition}

\begin{notation}
    Throughout the paper, we will use the phrase ``revlex-universal Gr\"obner basis'' to denote a set $\mathcal{G}$ of polynomials forming a Gr\"obner basis with respect to all possible degree reverse lexicographic orders (i.e., $\mathcal{G}$ is a revlex-Gr\"obner basis for any ordering of the variables).
\end{notation}

\begin{corollary} \label{cor:ehh}
Let $I \subseteq S = \field[x_1, \ldots, x_n]$ be a homogeneous ideal, let $R = S/I$ and let $X = \{x_1, \ldots, x_n\}$ be the set of variables of $S$. Assume that $I$ admits a revlex-universal Gr\"obner basis of quadrics. Then, for every $Y \subsetneq X$ and $x \in X \setminus (Y \cup \{x\})$, the colon ideal $(I + Y) :_S x$ is generated by linear forms modulo $I$. Thus, the colon ideal $\overline{Y} :_R \overline{x}$ is generated by linear forms.
\end{corollary}

\begin{remark}
Under the hypotheses of \Cref{cor:ehh}, Restuccia and Rinaldo claim in \cite[Theorem 1.10]{RR} that, since $I$ has a revlex-universal Gr\"obner basis of quadrics, then $R$ is strongly Koszul (with respect to the variables). Unfortunately, their proof contains a gap; namely, they claim that the colon ideals $\overline{Y} :_R \overline{x}$ we are interested in are not just generated by linear forms, but by variables. However, this is false in general: for a counterexample, take $S = \field[x,y]$ and $I = (x(x-y))$. Then $\overline{0} :_R \overline{x} = (\overline{x-y})$.
\end{remark}

\Cref{prop:ehh} and \Cref{cor:ehh} provide a sufficient condition for certain colon ideals to be linearly generated. In the same vein, we now wish to find a sufficient condition for the same colon ideals to be monomial (but not necessarily linear). To reach this goal, we need to introduce a class of homogeneous polynomials with a very specific kind of sparsity: we will call such polynomials \emph{tidy}.

\begin{definition}[tidy polynomial] \label{def:tidy}
Let $S = \field[x_1, \ldots, x_n]$ and let $f$ be a homogeneous polynomial in $S$. If each variable of $S$ divides at most one monomial in the support of $f$, we say that $f$ is \emph{tidy}. If a set $\mathcal{F} \subseteq S$ consists of tidy polynomials, we say that $\mathcal{F}$ is tidy.
\end{definition}

Tidiness is a natural feature of both monomial and (homogeneous) toric ideals, but we will soon see that it arises in other contexts as well. Recall that an integer matrix $A \in \mathbb{Z}^{m \times n}$ is called a \emph{configuration matrix} if the all-ones vector in $\mathbb{Z}^n$ can be written as a rational linear combination of the rows of $A$. The toric ideal $I_A$ is homogeneous if and only if $A$ is a configuration matrix: see, e.g., \cite[Proposition 3.3]{BinomialIdeals}.

\begin{example} \label{ex:tidy examples}
\phantom{.}
    \begin{enumerate}
        \item Every monomial is trivially tidy. In particular, the unique minimal monomial generating set for a monomial ideal $I$ is a tidy set (and a tidy universal Gr\"obner basis for $I$).   
        \item If $A \in \mathbb{Z}^{m \times n}$ is a configuration matrix, then the homogeneous toric ideal $I_A$ admits a tidy universal Gr\"obner basis. Indeed, by \cite[Corollary 4.4]{Sturmfels}, any reduced Gr\"obner basis for $I_A$ consists of binomials of the form $\mathbf{x}^{\mathbf{u}_+} - \mathbf{x}^{\mathbf{u}_-}$, for some suitable $\mathbf{u} \in \mathbb{Z}^n$ (here $\mathbf{u}_+$ and $\mathbf{u}_-$ are the unique elements of $\mathbb{N}^n$ that have disjoint support and are such that $\mathbf{u} = \mathbf{u}_+ - \mathbf{u}_-$). By construction, every variable of $S$ appears at most in one of $\mathbf{x}^{\mathbf{u}_+}$ and $\mathbf{x}^{\mathbf{u}_-}$, and thus the binomial $\mathbf{x}^{\mathbf{u}_+} - \mathbf{x}^{\mathbf{u}_-}$ is tidy.
        \item Let $M$ be a matrix whose entries are zeros or variables in the polynomial ring $S$. If every variable of $S$ appears in at most one entry on each row and on each column, then all $2$-minors and all $2$-permanents of $M$ are tidy. This happens in particular when $M$ is a sparse generic or a sparse generic symmetric matrix: i.e., a matrix obtained from a generic or generic symmetric matrix by setting to zero some of the entries (making sure to preserve symmetry in the generic symmetric case).
        \item Let $\Delta$ be a pure simplicial complex on $n$ vertices, $\field$ a field of characteristic zero, $S$ the polynomial ring $\field[x_i, z_G \mid i \in \{1, \ldots, n\},\  G \text{ facet of }\Delta]$. Consider the Artinian Gorenstein ring $A_{\Delta} = S/\mathrm{ann}(F_{\Delta})$ associated with the polynomial \[F_{\Delta} := \sum_{G \text{ facet of }\Delta}z_G \cdot \prod_{i \in G}x_i\] via Macaulay's inverse system (see \Cref{subsec:Macaulay's inverse system} and \Cref{rem:squarefree Macaulay in characteristic zero} therein). Then the ring $A_{\Delta}$ admits a tidy universal Gr\"obner basis. Indeed, one can use \cite[Proposition 8.1]{DV} to infer from \cite[Proposition 4.3]{DV} a system of tidy generators for the defining ideal of $A_{\Delta}$. Then, repeating verbatim\footnote{The proof of \cite[Proposition 4.5]{DV} lacks some details about the S-pairs obtained by picking a binomial and a monomial, but the result still stands.} the proof of \cite[Proposition 4.5]{DV}, one gets the desired universal Gr\"obner basis for the defining ideal of $A_{\Delta}$.
    \end{enumerate}
\end{example}

\begin{proposition} \label{prop:monomial colon ideals}
    Let $I \subseteq S = \field[x_1, \ldots, x_n]$ be a homogeneous ideal, let $R = S/I$ and let $\rev$ be the revlex order on $S$ such that $x_1 > x_2 > \ldots > x_n$. Assume that $I$ admits a tidy Gr\"obner basis with respect to $\rev$. Then, for every $i \in [n]$, the colon ideal $(I + (x_{i+1}, x_{i+2}, \ldots, x_n)) :_S x_i$ is generated by monomials modulo $I$. Thus, the colon ideal $(\overline{x}_{i+1}, \overline{x}_{i+2}, \ldots, \overline{x}_n) :_R \overline{x}_i$ is a monomial ideal.
\end{proposition}

\begin{proof}
Fix $i \in \{1, 2, \ldots, n\}$ and let $I' := I + (x_{i+1}, x_{i+2}, \ldots, x_n)$. It is our goal to show that $I' :_S x_i$ is generated by monomials modulo $I$. Let $f \in I' :_S x_i$. Since $\mathcal{G}$ is a Gr\"obner basis for $I$, there exists a unique remainder $r \in S$ such that $f = p + r$ for some $p \in I$ and $\init_{\rev}(g)$ does not divide any term in $r$ for any $g \in \mathcal{G}$. Being in $I$, the polynomial $p$ also belongs to $I':_S x_i$, and thus so does $r = f-p$. If $r=0$ there is nothing to prove, so assume otherwise. We want to show that each monomial in the support of $r$ lies in $I' :_S x_i$ as well. To achieve this, it is enough to prove that $\init_{\rev}(r)$ lies in $I' :_S x_i$, since then $r - \init_{\rev}(r) \in I' :_S x_i$ and we can repeat the process.
    
By \Cref{lem:revlex}(ii), it holds that $\init_{\rev}(I' :_S x_i) = (\init_{\rev}(I) + (x_{i+1}, x_{i+2}, \ldots, x_n)) :_S x_i$; hence, the monomial $\init_{\rev}(r) \cdot x_i$ lies in the monomial ideal $(\init_{\rev}(I) + (x_{i+1}, x_{i+2}, \ldots, x_n))$. 
If $\init_{\rev}(r) \cdot x_i \in (x_{i+1}, x_{i+2}, \ldots, x_n)$, it follows immediately that $\init_{\rev}(r) \in I' :_S x_i$. 
Assume now that $\init_{\rev}(r) \cdot x_i \notin (x_{i+1}, x_{i+2}, \ldots, x_n)$; it must then be that $\init_{\rev}(r) \cdot x_i \in \init_{\rev}(I)$. 
Then, since $\mathcal{G}$ is a Gr\"obner basis for $I$, there must exist $g \in \mathcal{G}$ such that $\init_{\rev}(g)$ divides $\init_{\rev}(r) \cdot x_i$. 
Due to the properties of the remainder $r$, we know that $\init_{\rev}(g)$ does not divide $\init_{\rev}(r)$, and thus $x_i$ must divide $\init_{\rev}(g)$. Since $\rev$ is the revlex order on $S$ such that $x_1 > x_2 > \ldots > x_n$ and $x_i$ divides $\init_{\rev}(g)$, it follows that there exist two homogeneous polynomials $h$ and $h'$ in $S$ such that 
\[g = \init_{\rev}(g) + x_i \cdot h + h',\]
    where each monomial in $h'$ contains at least one variable lower than $x_i$, which implies that $h'$ belongs to the ideal $(x_{i+1}, x_{i+2}, \ldots, x_n)$. Since $\mathcal{G}$ is tidy, the variable $x_i$ can divide at most one monomial in the support of $g$: this implies that $h = 0$, since $x_i$ is already featured in $\init_{\rev}(g)$. As a consequence, $\init_{\rev}(g) = g - h' \in I'$. Since $\init_{\rev}(g)$ divides $\init_{\rev}(r) \cdot x_i$, we have that $\init_{\rev}(r) \cdot x_i \in I'$ and hence $\init_{\rev}(r) \in I' :_S x_i$, as desired.
\end{proof}

\begin{corollary} \label{cor:tidy revlex UGB}
Let $I \subseteq S = \field[x_1, \ldots, x_n]$ be a homogeneous ideal, let $R = S/I$ and let $X = \{x_1, \ldots, x_n\}$ be the set of variables of $S$. Assume that $I$ admits a tidy revlex-universal Gr\"obner basis. Then, for every $Y \subseteq X$ and $x \in X \setminus (Y \cup \{x\})$, the colon ideal $(I + Y) :_S x$ is generated by monomials modulo $I$. Thus, the colon ideal $\overline{Y} :_R \overline{x}$ is generated by monomials.
\end{corollary}

\begin{remark}
If $I$ satisfies the hypotheses of \Cref{cor:tidy revlex UGB} (for instance, by \Cref{ex:tidy examples}, if $I$ is a monomial ideal or a homogeneous toric ideal), it follows that $R = S/I$ is strongly Koszul with respect to the variables if and only if the collection of all ideals of $R$ generated by variables forms a Koszul filtration\footnote{When the collection of all ideals of $R$ generated by variables forms a Koszul filtration, Ene, Herzog and Hibi say that $R$ is \emph{c-universally Koszul}: see \cite[p.~522]{EHH}.}. The ``only if'' is a consequence of strong Koszulness. To prove the ``if'' direction, pick $y_1, \ldots, y_k, x$ among the variables of $S$. A routine analysis of the Tor long exact sequence associated with the short exact sequence 
\[0 \to \frac{R}{(\overline{y}_1, \ldots, \overline{y}_k) :_R \overline{x}}(-1) \xrightarrow{\cdot \overline{x}} \frac{R}{(\overline{y}_1, \ldots, \overline{y}_k)} \to \frac{R}{(\overline{y}_1, \ldots, \overline{y}_k, \overline{x})} \to 0\]
yields that the colon ideal $(\overline{y}_1, \ldots, \overline{y}_k) :_R \overline{x}$ is linearly generated. Since $I$ has a tidy revlex-universal Gr\"obner basis by assumption, it follows that the colon ideal $(\overline{y}_1, \ldots, \overline{y}_k) :_R \overline{x}$ is actually generated by variables, as desired. Note that, under the additional hypothesis that the ring is toric, the equivalence between strong Koszulness with respect to the variables and $c$-universal Koszulness had already been stated by Murai \cite[Lemma 3.18]{MuraiKoszul}.
\end{remark}

We are now ready to state the main theorem of this section.

\begin{theorem} \label{thm:main theorem}
Let $I \subseteq S = \field[x_1, \ldots, x_n]$ be a homogeneous ideal with a tidy revlex-universal Gr\"obner basis of quadrics. Then $R = S/I$ is strongly Koszul with respect to the variables.
\end{theorem}

\begin{proof}
    Let $Y \subseteq \{x_1, \ldots, x_n\}$ and let $x \in \{x_1, \ldots, x_n\} \setminus Y$. Since $I$ has a revlex-universal Gr\"obner basis of quadrics, the colon ideal $\overline{Y} :_R \overline{x}$ is generated by linear forms by \Cref{cor:ehh}. In particular, $\overline{Y} :_R \overline{x}$ is \emph{minimally} generated by linear forms. Since $I$ has a tidy revlex-universal Gr\"obner basis, the colon ideal $\overline{Y} :_R \overline{x}$ admits a generating set made of monomials by \Cref{cor:tidy revlex UGB}. Extracting a minimal generating set from such a monomial generating set yields that $\overline{Y} :_R \overline{x}$ is (minimally) generated by variables, which proves the claim. 
\end{proof}

Since monomials are trivially tidy, \Cref{thm:main theorem} immediately implies the following well-known result (see, e.g., \cite[Theorem 3.15]{CDR}):
\begin{corollary} \label{cor:monomial strongly Koszul}
    Let $S = \field[x_1, \ldots, x_n]$ and let $I \subseteq S$ be a quadratic monomial ideal. Then $R = S/I$ is strongly Koszul with respect to the variables.
\end{corollary}

\Cref{thm:main theorem} also allows us to recover a result about toric rings featured in \cite[Theorem 2.7]{RR}, \cite[Corollary 1.4]{EHH} and \cite[Corollary 1.3 and Remark 1.5]{MatsudaOhsugi}:

\begin{corollary} \label{cor:toric strongly Koszul}
    Let $S = \field[x_1, \ldots, x_n]$, let $A \in \mathbb{Z}^{m \times n}$ be a configuration matrix and let $\field[A] = S/I_A$ be the (standard graded) toric ring associated with $A$. Assume that, for every choice of a total ordering of the variables of $S$, the reduced Gr\"obner basis of $I_A$ with respect to the associated revlex order consists of quadrics. Then the toric ring $\field[A]$ is strongly Koszul with respect to the variables.
\end{corollary}
\begin{proof}
As noted in \Cref{ex:tidy examples}(ii), every reduced Gr\"obner basis of the homogeneous toric ideal $I_A$ is tidy. Since by assumption every reduced Gr\"obner basis of $I_A$ with respect to a revlex order is quadratic, the hypotheses of \Cref{thm:main theorem} are met, and the claim follows.
\end{proof}

As noted by Matsuda and Ohsugi \cite[Examples 1.6 and 1.7]{MatsudaOhsugi}, the implication in \Cref{cor:toric strongly Koszul} cannot be reversed; indeed, there exist strongly Koszul toric rings that do not admit a revlex-universal Gr\"obner basis of quadrics. An example is the fourth Veronese subalgebra of $\field[t_1, t_2]$: see \Cref{rem:Hankel} below for an explanation. Such objects exist also if we further require that the ring is generated by squarefree monomials as a $\field$-algebra: an example proposed by Matsuda and Ohsugi is $\field[t_4, t_1t_4, t_2t_4, t_3t_4, t_1t_2t_4, t_1t_3t_4, t_2t_3t_4, t_1t_2t_3t_4]$.

We now provide an even smaller example of a strongly Koszul algebra where the hypotheses of \Cref{thm:main theorem} are not met. We will see more in \Cref{sec:not G-quadratic}.

\begin{remark} \label{rem:SK and no quadratic revlex-UGB}
    Let $S = \field[x_1,x_2,x_3]$ and consider $R = S/I$, where $I = (x_1x_3-x_2^2, x_2x_3, x_3^2)$. One checks that $R$ is strongly Koszul with respect to the variables, but the revlex initial ideal of $I$ with respect to the ordering $x_2 < x_1 < x_3$ has the cubic form $x_2^3$ among its minimal generators.
\end{remark}

It is a consequence of \Cref{lem:SK and quotients} that, if $R = S/I$ is strongly Koszul with respect to the variables, then so is any ring obtained from $R$ by quotienting out some variables. We now wish to show that the same kind of transfer holds for a stronger condition, namely, when the ideal $I$ admits a tidy revlex-universal Gr\"obner basis of quadrics.

\begin{lemma} \label{lem:revlex GB modulo variables}
Let $I \subseteq S = \field[x_1, \ldots, x_n]$ be a homogeneous ideal, let $\rev$ be the revlex order on $S$ such that $x_1 > x_2 > \ldots > x_n$ and let $\mathcal{G}$ be a Gr\"obner basis of $I$ with respect to $\rev$. Let $S' = \field[x_1, \ldots, x_i]$, let $\pi_i$ be the natural projection $S \twoheadrightarrow S'$ and let $\widetilde{\mathcal{G}} := \{\pi_i(g) \mid g \in \mathcal{G},\ g \notin \ker(\pi_i)\}$. Then $\widetilde{\mathcal{G}}$ is a Gr\"obner basis of $\pi_i(I)$ with respect to the revlex order $<'_{\mathrm{rev}}$ on $S'$ such that $x_1 > x_2 > \ldots > x_i$. Moreover, if $\mathcal{G}$ is quadratic (respectively, tidy), then so is $\widetilde{\mathcal{G}}$.
\end{lemma}
\begin{proof}
It is clear that quadraticity and tidiness are passed on from the set $\mathcal{G}$ to the set $\widetilde{\mathcal{G}}$. Let us prove that $\widetilde{\mathcal{G}}$ is a Gr\"obner basis of $\pi_i(I)$ with respect to the revlex order $<'_{\mathrm{rev}}$.

\noindent We need to show that $\init_{<'_{\mathrm{rev}}}(\pi_i(I)) = (\{\init_{<'_{\mathrm{rev}}}(\pi_i(g)) \mid g \in \mathcal{G},\  g \notin \ker(\pi_i)\})$. By \Cref{lem:revlex and projection}(i), one has that \[
\begin{split}(\{\init_{<'_{\mathrm{rev}}}(\pi_i(g)) \mid g \in \mathcal{G},\  g \notin \ker(\pi_i)\}) &= (\{\pi_i(\init_{<_{\mathrm{rev}}}(g)) \mid g \in \mathcal{G},\  g \notin \ker(\pi_i)\})\\ &= \pi_i(\{\init_{<_{\mathrm{rev}}}(g)) \mid g \in \mathcal{G},\  g \notin \ker(\pi_i)\})\\ &= \pi_i(\{\init_{<_{\mathrm{rev}}}(g)) \mid g \in \mathcal{G}\}).
\end{split}\]
Since $\mathcal{G}$ is a Gr\"obner basis of $I$ with respect to $<_{\mathrm{rev}}$, one has that $\pi_i(\{\init_{<_{\mathrm{rev}}}(g)) \mid g \in \mathcal{G}\}) = \pi_i(\init_{<_{\mathrm{rev}}}(I))$; by \Cref{lem:revlex and projection}(ii), one has that $\pi_i(\init_{<_{\mathrm{rev}}}(I)) = \init_{<'_{\mathrm{rev}}}(\pi_i(I))$, which concludes the proof.
\end{proof}

\begin{corollary} \label{cor:revlex-UGB modulo variables}
    Let $I \subseteq S = \field[x_1, \ldots, x_n]$ be a homogeneous ideal, let $X = \{x_1, \ldots, x_n\}$ be the set of variables of $X$ and let $Y \subseteq X$. Let $T = \field[x \mid x \in X \setminus Y]$ and let $\pi$ be the natural projection $S \twoheadrightarrow T$. Let $\mathcal{G}$ be a revlex-universal Gr\"obner basis of $I$, and let $\widetilde{\mathcal{G}} := \{\pi(g) \mid g \in \mathcal{G},\ g \notin \ker(\pi)\}$. Then $\widetilde{\mathcal{G}}$ is a revlex-universal Gr\"obner basis of $\pi(I)$. Moreover, if $\mathcal{G}$ is quadratic (respectively, tidy), then so is $\widetilde{\mathcal{G}}$.
\end{corollary}
\begin{proof}
The statements about quadraticity and tidiness are clear. Now pick any revlex order $\prec_T$ in $T$; it is our goal to show that $\widetilde{\mathcal{G}}$ is a Gr\"obner basis with respect to $\prec_T$. Let $<_S$ be a revlex order in $S$ such that:
\begin{itemize}
    \item if $y \in Y$, $z \in X \setminus Y$, then $y <_S z$;
    \item if $z, z' \in X \setminus Y$, then $z <_S z'$ if and only if $z \prec_T z'$.
\end{itemize}
Since $\mathcal{G}$ is a revlex-universal Gr\"obner basis of $I$ by assumption, in particular it is a Gr\"obner basis of $I$ with respect to $<_S$; hence, by \Cref{lem:revlex GB modulo variables}, we have that $\widetilde{\mathcal{G}}$ is a Gr\"obner basis of $\pi(I)$ with respect to $\prec_T$, which ends the proof.
\end{proof}

\section{Strong Koszulness for quadric hypersurfaces and ideals of 2-minors} \label{sec:applications}
In this section we will use \Cref{thm:main theorem} to certify strong Koszulness for certain classes of standard graded $\field$-algebras. As a warm-up, we begin by addressing the case of a quadric hypersurface. This is an immediate consequence of \Cref{thm:main theorem} when the characteristic of $\field$ is not $2$, but requires some extra work otherwise.

\begin{proposition} \label{prop:quadric hypersurface}
    Let $S = \field[x_1, \ldots, x_n]$ and let $f \in S_2$. Then $S/(f)$ is strongly Koszul.
\end{proposition}
\begin{proof}
Assume first that $\mathrm{char}(\field) \neq 2$. Then, after a linear change of coordinates, the quadric $f$ can be written as $\sum_{i=1}^{n}\lambda_i x_i^2$, where each $\lambda_i$ is a (possibly null) element of $\field$: see, e.g.,~\cite[Section IV.1.4]{SerreArithmetic} or \cite[Proposition 7.29]{QuadraticForms}. The claim then follows from \Cref{thm:main theorem}.

Let us now assume that $\mathrm{char}(\field) = 2$. Call $X$ the set of variables in $S$. By \cite[Proposition 7.31]{QuadraticForms}, there exists a linear change of coordinates sending the quadric $f$ into
\[\sum_{i=1}^{m}(a_iv_i^2 + v_iw_i + b_iw_i^2) + \sum_{k=1}^{t}c_kz_k^2,\]
where all the $a_i, b_i, c_k$ are elements of $\field$, each $c_k$ is nonzero, and all the $v_i, w_i, z_k$ are distinct elements of $X$. If some $b_i$ is zero, then we apply the change of coordinates $w_i \mapsto w_i+a_iv_i$ to transform the part $a_iv_i^2+v_iw_i$ of the quadric into $v_iw_i$. (Recall that we are working in characteristic $2$, so signs do not matter.) Analogously, if some $a_i$ is zero, we apply the change of coordinates $v_i \mapsto v_i+b_iw_i$ to the same effect. We can hence assume without loss of generality that our quadric is
\[g := \sum_{i=1}^{s}(a_iv_i^2 + v_iw_i + b_iw_i^2) + \sum_{j=s+1}^{m}v_jw_j + \sum_{k=1}^{t}c_kz_k^2,\]
where all the $a_i, b_i, c_k$ are now nonzero elements of $\field$. Now let $I = (g)$ and $R = S/I$. We need to prove that, for any $Y \subsetneq X$ and $x \in X \setminus Y$, the colon ideal $\overline{Y} :_R \overline{x}$ (that is isomorphic to $(I':_S x)/I$, where $I' := I + Y$) is generated by variables. 
We now follow closely the argument in the proof of \Cref{thm:main theorem}. Since $\mathcal{G} := \{g\}$ is trivially a revlex-universal Gr\"obner basis of quadrics for the principal ideal $I$, the colon ideal $\overline{Y} :_R \overline{x}$ is generated by linear forms. 
We now need to prove that $\overline{Y} :_R \overline{x}$ is generated by variables. If $\overline{Y} :_R \overline{x} = (\overline{0})$ there is nothing to prove, so assume otherwise and pick a nonzero homogeneous linear form $\ell \in I' :_S x$. 
Fix a revlex term order $\rev$ on $S$ such that $y < x < z$ for every $y \in Y$ and $z \in X \setminus (Y \cup \{x\})$. Recall that, by \Cref{lem:revlex}(ii), it holds that $\init_{\rev}(\ell) \cdot x \in (\init_{\rev}(I) + Y)$; we want to show that $\init_{\rev}(\ell) \in I':_S x$. 
If $\init_{\rev}(\ell) \cdot x \in (y \mid y \in Y)$, we are done; hence, for the rest of the proof we assume that $\init_{\rev}(\ell)\cdot x \notin (y \mid y \in Y)$ and $\init_{\rev}(\ell) \cdot x\in \init_{\rev}(I)$, which implies that $\init_{\rev}(\ell) \cdot x = \init_{\rev}(g)$. 
If $x$ appears in at most one monomial in $g$, we conclude as in the proof of \Cref{thm:main theorem}. Assume otherwise. Then, without loss of generality, we can assume that $x = v_1$. Since $\init_{\rev}(\ell)\cdot v_1 = \init_{\rev}(g)$, it must be that $\init_{\rev}(g) = a_1v_1^2$. Due to how $<_{\mathrm{rev}}$ was chosen, one can have that $\init_{\rev}(g) = a_1v_1^2$ only if every monomial in $g-a_1v_1^2$ is divisible by a variable in $Y$; as a consequence, $g-\init_{\rev}(g) \in (y \mid y \in Y)$, which in turn implies that $\init_{\rev}(g) \in I'$. It follows that $\init_{\rev}(\ell) \cdot x = \init_{\rev}(g) \in I'$, which ends the proof.
\end{proof}

We devote the rest of this section to analyze ideals of $2$-minors of matrices filled with variables and zeros.  To do so, we introduce some useful notation following Conca and Welker \cite[Section 7]{ConcaWelker}.

\begin{notation} \label{not:concawelker}
\phantom{.}
\begin{enumerate}
    \item If $m, n$ are positive integers, let $\Sgen = \field[x_{ij} \mid i \in [m], j \in [n]]$ and let $\Xgen$ be the $m \times n$ matrix whose $(i,j)$-th entry is $x_{ij}$. If $G$ is a bipartite graph on $[m] \sqcup [\tilde{n}]$ with edge set $E$, we will denote by $\Xgen_G$ the matrix whose $(i,j)$-th entry is $0$ if $\{i, \tilde{j}\} \in E$ and $x_{ij}$ otherwise.
    \item If $n$ is a positive integer, let $\Ssym = \field[x_{ij} \mid 1 \leq i \leq j \leq n]$ and let $\Xsym$ be the $n \times n$ symmetric matrix whose $(i,j)$-th entry (with $i \leq j$) is $x_{ij}$. If $H$ is a graph (possibly with loops) on $[n]$, we will denote by $\Xsym_H$ the matrix obtained from $\Xsym$ by setting to zero the $(i,j)$-th and $(j,i)$-th entry whenever $\{i,j\}$ (with $i \neq j$) is an edge of $H$, and the $(i,i)$-th entry whenever $H$ has a loop in $i$.
    \item If $n$ is a positive integer, let $\Ssk = \field[x_{ij} \mid 1 \leq i < j \leq n]$ and let $\Xsk$ be the $n \times n$ skew-symmetric matrix whose $(i,j)$-th entry (with $i < j$) is $x_{ij}$.
\end{enumerate}
\end{notation}

The next result was first claimed by Restuccia and Rinaldo in \cite[Theorem 3.6]{RR}. Since the proof argument presented in \cite{RR} is not clear to us, we produce an independent proof below.

\begin{proposition} \label{prop:symmetric revlex-universal 2-minors}
    In the setting of \Cref{not:concawelker}(ii), one has that the $2$-minors of the generic symmetric matrix $\Xsym$ form a tidy quadratic revlex-universal basis for the ideal $I_2(\Xsym)$ they generate. 
\end{proposition}

\begin{convention}
    In the proof of \Cref{prop:symmetric revlex-universal 2-minors}, if $i_1$ and $i_2$ are possibly coincident integers in $\{1,2,\ldots, n\}$, we will write either $x_{i_1j_1}$ or $x_{j_1i_1}$ to denote the variable $x_{\min\{i_1,j_1\},\max\{i_1,j_1\}}$ in $\Ssym$. Also, we will write $[i_1i_2 \ | \ j_1j_2]$ to denote the $2$-minor of $\Xsym$ obtained by picking row indices $i_1,i_2$ and column indices $j_1,j_2$ (or vice versa, since $\Xsym$ is symmetric), and will abuse notation by considering such a minor only up to a global sign (i.e., $[i_1i_2 \ | \ j_1j_2]$ can be either $x_{i_1j_1}x_{i_2j_2}-x_{i_1j_2}x_{i_2j_1}$ or $-x_{i_1j_1}x_{i_2j_2}+x_{i_1j_2}x_{i_2j_1}$).
\end{convention}

\begin{proof}[Proof of \Cref{prop:symmetric revlex-universal 2-minors}]
 Note that a $2$-minor of $\Xsym$ contains the variable $x_{ij}$ (with $i,j$ possibly coincident) precisely when it is obtained by picking $i$ among the row indices and $j$ among the column indices (or vice versa). We want to prove that the S-polynomial of any two distinct 2-minors $f$ and $g$ with non-coprime initial terms reduces to zero. 
 Say these 2-minors both contain the variable $x_{ij}$ (with $i,j$ possibly coincident) in their leading terms. Then there exist indices $k,a$ (possibly coincident, but both different from $i$) and $\ell,b$ (possibly coincident, but both different from $j$) such that \[f = [ik \ |\  j\ell] = \underline{x_{ij}x_{k\ell}}-x_{i\ell}x_{kj}\quad \text{and} \quad g = [ia \ | \ jb] = \underline{x_{ij}x_{ab}}-x_{ib}x_{aj},\] where we underlined the respective leading terms. 
 It follows that, in the revlex order we are considering, at least one of the variables $x_{i\ell}$ and $x_{kj}$ is lower than $x_{k\ell}$. Analogously, one of the variables $x_{ib}$ and $x_{aj}$ is lower than $x_{ab}$.\\
If $a = k$ and $b=\ell$, then $f$ and $g$ are the same polynomial. If $a = \ell$ and $b = k$, then $f$ and $g$ have the same leading term and the S-polynomial of $f$ and $g$ is $x_{ik}x_{\ell j} - x_{i\ell}x_{kj} = [ij \ | \ k\ell]$, which is either zero (if $i=j$ or $k=\ell$) or reduces to zero.\\
For the rest of the proof assume that $\{k,\ell\} \neq \{a,b\}$, and thus $x_{k\ell} \neq x_{ab}$. Then the S-polynomial of $f$ and $g$ is \[S(f,g) = x_{ib}x_{aj}x_{k\ell} - x_{i\ell}x_{kj}x_{ab}.\]
If $a=k$ (and thus $b \neq \ell$), then $S(f,g) = x_{kj}(x_{ib}x_{k\ell}-x_{i\ell}x_{kb}) = x_{kj}\cdot [ik\ | \ \ell b]$ reduces to zero and we are done.\\
If $b=\ell$ (and thus $a \neq k$), then $S(f,g) = x_{i\ell}(x_{aj}x_{k\ell}-x_{a\ell}x_{kj}) = x_{i\ell} \cdot [ ka\ | \ j\ell]$ reduces to zero and we are done.\\
For the rest of the proof assume thus that $a\neq k$, $b \neq \ell$ and $(a,b) \neq (\ell,k)$. One of $x_{ib}, x_{aj}, x_{i\ell}, x_{kj}$ is the lowest variable among those appearing in $S(f,g)$: without loss of generality, say it is $x_{i\ell}$. Then the leading term of $S(f,g)$ is $x_{ib}x_{aj}x_{k\ell}$ and the leading term of the $2$-minor $[ik \ | \ \ell b] = x_{ib}x_{k\ell} - x_{i\ell}x_{kb}$ is $x_{ib}x_{k\ell}$. Using this $2$-minor, one can then reduce $S(f,g)$ to $x_{i\ell}(x_{kb}x_{aj}-x_{kj}x_{ab}) = x_{i\ell} \cdot [ka \ | \ jb]$, which reduces to zero.
\end{proof}

Since the ideal of $2$-minors of the generic symmetric matrix of size $n$ is isomorphic to the defining (toric) ideal of the second Veronese subalgebra of $\field[y_1, \ldots, y_n]$, \Cref{prop:symmetric revlex-universal 2-minors} recovers the fact that such an algebra is strongly Koszul with respect to the variables: see \cite[Proposition 2.3(a)]{HHR} and \cite[Theorem 13]{CDR}. In fact, we can say more:

\begin{theorem} \label{thm:2-minors}
    In the setting of \Cref{not:concawelker}, one has the following results.
    \begin{enumerate}
    \item For any bipartite graph $G$, the ideal $I_2(\Xgen_G)$ admits a tidy quadratic revlex-universal Gr\"obner basis, and hence defines a strongly Koszul algebra with respect to the variables.
    \item For any graph $H$, the ideal $I_2(\Xsym_H)$ admits a tidy quadratic revlex-universal Gr\"obner basis, and hence defines a strongly Koszul algebra with respect to the variables.
    \end{enumerate}
\end{theorem}

Before moving to the proof of \Cref{thm:2-minors}, we note that the analogous statement for $4$-Pfaffians does not hold.

\begin{remark}
The $4$-Pfaffians of the generic $n \times n$ skew-symmetric matrix generate the defining ideal $I_n$ of the coordinate ring $R_n$ of the Grassmannian $\mathrm{Gr}(2,n)$. It is well-known that the $4$-Pfaffians themselves form a (tidy) Gr\"obner basis of quadrics for $I_n$ in any characteristic (see, e.g., \cite[Theorem 5.1]{HerzogTrung} or \cite[Theorem 2.1]{JonssonWelker} for two different term orders). The ring $R_4$ is defined by a single tidy quadratic polynomial and hence, by \Cref{cor:tidy revlex UGB}, is strongly Koszul with respect to the variables. However, as soon as $n \geq 5$, the ring $R_n$ is not strongly Koszul with respect to the variables; to see this, let $J := I_5 + (x_{23}, x_{35}, x_{45})$ and let $R'_5 := R_5/\overline{J}$. One checks that $0 :_{R'_5} \overline{x}_{24} = (\overline{x}_{13}\overline{x}_{15})$. Since the ideal $J$ is generated by monomials and differences of monomials, such an equality holds in any characteristic. An application of \Cref{lem:SK and quotients} then yields that $R_n$ is not strongly Koszul with respect to the variables for any $n \geq 5$.
\end{remark}

\begin{proof}[Proof of \Cref{thm:2-minors}]
 Note that the claims about the existence of a tidy quadratic revlex-universal Gr\"obner basis do not depend on the properties of the field $\field$. Indeed, since the generators of $I_2(\Xgen_G)$ and $I_2(\Xsym_H)$ are monomials and differences of monomials, the Buchberger algorithm works the same for any choice of $\field$.
\begin{enumerate}
\item It is known that the $2$-minors of the generic matrix $\Xgen$ form a (tidy) quadratic revlex-universal Gr\"obner basis for $I_2(\Xgen)$, see \cite[Proposition 5.3.8]{BCRV}. Since \[\frac{\Sgen}{I_2(\Xgen_G)} = \frac{\Sgen}{I_2(\Xgen) + (x_{ij} \mid \{i,\tilde{j}\} \in E(G))},\]
it follows from \Cref{cor:revlex-UGB modulo variables} that $I_2(\Xgen_G)$ also admits a tidy quadratic revlex-universal Gr\"obner basis. Thus, the ring defined by $I_2(\Xgen_G)$ is strongly Koszul with respect to the variables by \Cref{thm:main theorem}.

\item By \Cref{prop:symmetric revlex-universal 2-minors}, the $2$-minors of the generic symmetric matrix $\Xsym$ form a (tidy) quadratic revlex-universal Gr\"obner basis of $I_2(\Xsym)$. Noting that 
\[\frac{\Ssym}{I_2(\Xsym_H)} = \frac{\Ssym}{I_2(\Xsym) + (x_{ij} \mid \{i,j\} \in E(H))},\]
we finish the proof in the same way as part (i).
\end{enumerate}
\end{proof} 

\begin{remark}[the generic Hankel case] \label{rem:Hankel}
Let $S = \field[x_1, \ldots, x_{n+1}]$ and let $M$ be the generic Hankel matrix with $m$ rows and $n-m+2$ columns (with $2 \leq m \leq n$), i.e., the $(i,j)$-th entry of $M$ is $x_{i+j-1}$. It is known that $I_2(M) \cong I_2(N)$, where $N$ is the generic Hankel matrix with $2$ rows and $n$ columns (see, e.g., \cite{Watanabe}). Moreover, $S/I_2(N)$ is isomorphic to the $n$-th Veronese subalgebra of $\field[y,z]$, and thus is strongly Koszul with respect to the variables by \cite[Proposition 2.3(a)]{HHR} or \cite[Theorem 13]{CDR}.

However, the ideal $I_2(N)$ does not admit a quadratic revlex-universal Gr\"obner basis as soon as $n \geq 3$. To see this, just set to zero all the variables $x_j$ with $j \geq 4$ and note that $S/(I_2(N) + (x_j \mid j \geq 4)) \cong \field[x_1,x_2,x_3]/(x_1x_3-x_2^2, x_2x_3, x_3^2)$. As we saw in \Cref{rem:SK and no quadratic revlex-UGB}, the ideal $(x_1x_3-x_2^2, x_2x_3, x_3^2)$ does not admit a quadratic revlex-universal Gr\"obner basis. By \Cref{cor:revlex-UGB modulo variables}, the ideal $I_2(N)$ does not admit a quadratic revlex-universal Gr\"obner basis either.
\end{remark}

\begin{remark}
If $M$ is a matrix whose entries are variables and zeros, and no variable of $S$ appears more than once on each row and on each column of $M$, then the $2$-minors of $M$ form a tidy set (see \Cref{ex:tidy examples}(iii)). One might wonder whether this is the ``right'' setting for the results about strong Koszulness we have found so far. However, the ideals of $2$-minors of such matrices do not even define Koszul algebras in general! The case where the matrix $M$ has two rows was studied by Nguyen, Thieu and Vu in \cite{2xe}. It follows from their work that, if $S=\mathbb{C}[x,y,z,t]$ and \[M = \begin{bmatrix}0 & x & y & z\\x & y & 0 & t\end{bmatrix},\]
the ring $S/I_2(M)$ is not Koszul.
\end{remark}

\section{Maximal minors, maximal Pfaffians and apolarity} \label{sec:apolarity and strong Koszulness}

In this section we show that certain determinantal objects, when viewed through the looking-glass of Macaulay's inverse system, are a source of strongly Koszul algebras. This extends previous work by Shafiei \cite{Sha}, who focused on finding explicit generators and Gr\"obner bases for such apolar ideals.

\subsection{Macaulay's inverse system} \label{subsec:Macaulay's inverse system}
In this section we collect some facts about a tool that will come in handy in the rest of the paper: Macaulay's inverse system. Our treatment follows closely \cite[Section 3]{MasutiTozzo}.

\subsubsection{The contraction action} Let $\field$ be any field, let $n \in \NN$ and consider the polynomial ring $S = \field[x_1, \ldots, x_n]$. Let $\mathcal{D}_i := \mathrm{Hom}_\field(S_i, \field)$ for every $i \in \NN$, and consider $\mathcal{D} := \bigoplus_{i \in \NN}\mathcal{D}_i$ (in other words, $\mathcal{D}$ is the graded dual module of $S = \mathrm{Sym}^{\bullet}(\field^n)$). We call $\mathcal{D}$ the \emph{divided power algebra}, but for the present paper we will not need to specify its multiplicative structure; we refer the interested reader to, e.g., \cite[Appendix A2.4]{Eisenbud}. In what follows we will denote by $\{X_1, \ldots, X_n\}$ the $\field$-basis of $\mathcal{D}_1$ that is dual to the $\field$-basis $\{x_1, \ldots, x_n\}$ of $S_1$. The ring $\mathcal{D}$ admits an $S$-module structure, called the \emph{contraction} \emph{action}, defined for every $\boldsymbol{\alpha}, \boldsymbol\beta \in \NN^n$ by
\[\mathbf{x}^{\boldsymbol\alpha} \circ_{\mathrm{ctr}} \mathbf{X}^{\boldsymbol\beta} := \begin{cases}\mathbf{X}^{\boldsymbol\beta-\boldsymbol\alpha} & \textrm{ if }\boldsymbol\beta \geq \boldsymbol\alpha \textrm{ componentwise}\\ 0 & \textrm{ else}\end{cases}\]
and then extended $\field$-linearly.

Given an ideal $I$ of $S$ (or, in other words, an $S$-submodule $I$ of $S$), we can associate with it the $S$-submodule $I^{\perp}$ of $\mathcal{D}$ defined as
\[I^{\perp} := \{F \in \mathcal{D} \mid f \circ_{\mathrm{ctr}} F = 0 \textrm{ for every }f \in I\}.\]
(Note that $I^{\perp}$ needs not be finitely generated as an $S$-submodule of $\mathcal{D}$; for an explicit example, take $I = (x^2, xy) \subseteq S = \field[x,y]$.)

Vice versa, given a (not necessarily finitely generated) $S$-submodule $M$ of $\mathcal{D}$, we can associate with it the ideal $\mathrm{ann}_S(M)$ of $S$ defined as
\[\mathrm{ann}_S(M) := \{f \in S \mid f \circ_{\mathrm{ctr}} F = 0 \textrm{ for every }F \in M\}.\]

The ideal $\mathrm{ann}_S(M)$ is sometimes called the \emph{apolar ideal} to $M$. One can prove that $\mathrm{ann}_S(I^{\perp}) = I$ for every ideal $I \subseteq S$ and $\mathrm{ann}_S(M)^{\perp} = M$ for every $S$-submodule $M \subseteq \mathcal{D}$. A result dating back to Macaulay exploits this duality (which in modern terms can be seen as a special case of Matlis duality) to relate Artinian level quotients of $S$ to finitely generated $S$-submodules of $\mathcal{D}$. More precisely, $(-)^{\perp}$ and $\mathrm{ann}_S(-)$ yield a bijection between those homogeneous ideals $I \subseteq S$ such that $S/I$ is an Artinian level ring of socle degree $s$ and type $\tau$ and those graded $S$-submodules of $\mathcal{D}$ finitely generated by $\tau$ linearly independent homogeneous polynomials of degree $s$. In particular, Artinian Gorenstein quotients of $S$ correspond to cyclic $S$-submodules of $\mathcal{D}$.

When $M$ is graded, then so is the ring $S/\mathrm{ann}_S(M)$, and one can access the Hilbert function of the latter by noting that $\mathrm{HF}(S/\mathrm{ann}_S(M), s) = \dim_{\field}M_s$. Such a result can be proved putting together the proof of \cite[Proposition 2.5]{Geramita} (adapted to the contraction setting) and the fact that $(\mathrm{ann}_S(M))^{\perp} = M$.

\subsubsection{The differentiation action} \label{subsec:differentiation}
For this subsection, let $\field$ have characteristic zero, $S = \field[x_1, \ldots, x_n]$ and $\mathcal{S} = \field[X_1, \ldots, X_n]$. The ring $\mathcal{S}$ admits an $S$-module structure (different from the standard multiplicative one), called the \emph{differentiation action}, defined for every $\boldsymbol{\alpha}, \boldsymbol\beta \in \NN^n$ by
\[\mathbf{x}^{\boldsymbol\alpha} \circ_{\mathrm{diff}} \mathbf{X}^{\boldsymbol\beta} := \begin{cases}\alpha_1! \cdot \alpha_2! \cdot \ldots \cdot \alpha_n! \cdot \mathbf{X}^{\boldsymbol\beta-\boldsymbol\alpha} & \textrm{ if }\boldsymbol\beta \geq \boldsymbol\alpha \textrm{ componentwise}\\ 0 & \textrm{ else}\end{cases}\]
In other words, each $x_i$ acts as the differential operator $\frac{\partial}{\partial X_i}$. We can now set up the $(-)^{\perp}$ and $\mathrm{ann}_S(-)$ operators analogously to the previous subsection, but using $S$-submodules of $\mathcal{S}$ with respect to the differentiation action. Macaulay's duality stays in place just like in the previous setting.

\begin{remark} \label{rem:squarefree Macaulay in characteristic zero}
    Let $\mathrm{char}(\field) = 0$, let $F_1, \ldots, F_{\ell}$ be polynomials in $\mathcal{D}$ (respectively, in $\mathcal{S}$) obtained as $\field$-linear combinations of squarefree monomials of the same degree and let $M$ be the $S$-submodule of $\mathcal{D}$ (respectively, of $\mathcal{S}$) generated by $F_1, \ldots, F_{\ell}$. Since no powers appear in the supports of $F_1, \ldots, F_{\ell}$, no powers can appear in any element of $M$; as a consequence, $\mathrm{ann}_S(M)$ is the same polynomial ideal of $S$ regardless of whether we use the contraction action on $\mathcal{D}$ or the differentiation action on $\mathcal{S}$.
\end{remark}
\begin{remark} \label{rem:differentiation and contraction for squarefree monomials}
    Let $\field$ be a field of arbitrary characteristic. If $F \in \mathcal{D}$ is a $\field$-linear combination of squarefree monomials and $y$ is a variable of $S$, acting on $F$ via contraction by $y$ or via the differential operator $\frac{\partial}{\partial y}$ yields the same result. For this reason, with an abuse of notation, when acting on $\field$-linear combinations of squarefree monomials we will often use the differentiation notation even when $\field$ has positive characteristic. In particular, if $M$ is an $S$-submodule of $\mathcal{D}$ generated by homogeneous $\field$-linear combinations of squarefree monomials, we can compute the graded structure of $M$ and the Hilbert series of $S/\mathrm{ann}_S(M)$ by considering the $\field$-linear spans of the appropriate partial derivatives of the generators of $M$.
\end{remark}

\subsection{Maximal minors and apolarity}
\begin{notation}
    In this section, when $Y$ is a matrix with $m$ rows and $n$ columns, $\Ic \subseteq [m]$ and $\Jc \subseteq [n]$, we will denote by $Y_{\Ic,\Jc}$ the matrix obtained from $Y$ by deleting all rows indexed by elements of $\Ic$ and all columns indexed by elements of $\Jc$. 
\end{notation}

If $Z$ is a square matrix of size $n$ and $i \in [n]$, the Laplace expansion of the determinant of $Z$ with respect to the $i$-th row yields that
\begin{equation} \label{eq:Laplace expansion}
\det(Z) = \sum_{j=1}^{n}(-1)^{i+j} \cdot z_{ij} \cdot \det(Z_{\{i\},\{j\}}).
\end{equation}

Following \Cref{not:concawelker}, let $\field$ be any field and $\Xgen$ be a generic $m \times n$ matrix (where $m \leq n$) with entries in the polynomial ring $\Sgen = \field[x_{ij} \mid 1 \leq i \leq m, 1 \leq j \leq n]$. 

We recall a definition from \cite[Introduction]{HPPRS}:
\begin{definition}
A \emph{generalized $2 \times 2$-permanent} (or \emph{generalized $2$-permanent}) of $\Xgen$ is the permanent of a $2 \times 2$ submatrix \[\begin{bmatrix}x_{i_1,j_1} & x_{i_1,j_2}\\x_{i_2,j_1} & x_{i_2,j_2}\end{bmatrix}\] of $\Xgen$, where $i_1$ and $i_2$ are not necessarily distinct, and likewise for $j_1$ and $j_2$.
\end{definition}

Consider the $\Sgen$-module $M$ (with respect to the contraction action) finitely generated by the maximal minors of $\Xgen$, and let $I = \mathrm{ann}(M)$ be the apolar ideal to $M$. The ideal $I$ defines a Gorenstein ring when $m=n$ and a level non-Gorenstein ring when $m<n$. 
For $m=n$, Shafiei found the generators of $I$ (in terms of permanents and ``unacceptable monomials'') and showed they form a quadratic Gr\"obner basis \cite[Theorem 2.12, Theorem 5.3]{Sha}. We wish to prove that, for $m \leq n$, the generators of $I$ are up to multiplicative constant the generalized $2$-permanents of $\Xgen$\footnote{When the characteristic of $\field$ is zero, this fact can be quickly proved by representation-theoretic methods. We thank Alessio Sammartano for this observation.} when $\mathrm{char}(\field) \neq 2$, and the ideal $I$ defines a strongly Koszul algebra in any characteristic. To do so, let us first determine the Hilbert function of $\Sgen/I$. Thanks to the properties of Macaulay's inverse system and \Cref{rem:differentiation and contraction for squarefree monomials}, this can be gauged by analyzing the partial derivatives of the maximal minors of $\Xgen$. 

\begin{lemma} \label{lem:Hilbert generic}
    For every $0 \leq s \leq m$, one has that $M_s = \left\langle s\text{-minors of }\Xgen \right\rangle$ and, thus, $\HF(\Sgen/I, s) = \binom{m}{s}\binom{n}{s}$.
\end{lemma}
\begin{proof}
    Fix $0 \leq s \leq m$. Let $i_1, \ldots, i_{m-s} \in [m]$ (not necessarily distinct), $\Ic = \{i_1, \ldots, i_{m-s}\}$, $j_1, \ldots, j_{m-s} \in [n]$ (not necessarily distinct), $\Jc = \{j_1, \ldots, j_{m-s}\}$. Moreover, let $\Hc \subseteq [n]$ be such that $|\Hc|=n-m$, and note that $\Xgen_{\varnothing,\Hc}$ is a square matrix of size $m$. It is a direct consequence of the Laplace expansion \eqref{eq:Laplace expansion} that 
    \[
    \frac{\partial^{m-s} (\det(\Xgen_{\varnothing,\Hc}))}{\partial x_{i_1, j_1} \partial x_{i_2, j_2} \ldots \partial x_{i_{m-s}, j_{m-s}}} = \begin{cases}\pm \det(\Xgen_{\Ic, \Hc \cup \Jc}) & |\Ic|=m-s,\ |\Hc \cup \Jc|=n-s\\ 0 & \text{otherwise}.\end{cases}
    \]
    Thus, $M_s$ is generated as a $\field$-vector space by all the $s$-minors of $\Xgen$. Since the collection of $s$-minors of $\Xgen$ is $\field$-linearly independent, it follows that 
    \[\HF(\Sgen/I, s) = \dim_{\field}M_s = \binom{m}{s}\binom{n}{s}.\]
\end{proof}

\begin{proposition} \label{prop:generic apolarity}
    Let $M$ be the $\Sgen$-module (with respect to the contraction action) generated by the maximal minors of $\Xgen$ and let $I = \mathrm{ann}(M)$. Then the set $\mathcal{G}$ consisting of the following polynomials \begin{itemize}
    \item {\makebox[3cm]{$x_{ik}^2$\hfill}} for $i \in [m],\ k \in [n]$;
    \item {\makebox[3cm]{$x_{ik}x_{jk}$\hfill}} for $i,j \in [m], \ k \in [n], \ i \neq j$;
    \item {\makebox[3cm]{$x_{ik}x_{i\ell}$\hfill}} for $i \in [m],\ k,\ell \in [n],\ k \neq \ell$;
    \item {\makebox[3cm]{$x_{ik}x_{j\ell}+x_{i\ell}x_{jk}$\hfill}} for $i,j \in [m], \ k,\ell \in [n],\ i \neq j, \ k \neq \ell$
    \end{itemize} is a tidy revlex-universal Gr\"obner basis  of $\binom{m+1}{2}\binom{n+1}{2}$ quadrics for $I$. As a consequence, the ideal $I$ defines a strongly Koszul algebra with respect to the variables.
\end{proposition}

\begin{proof}
By \Cref{lem:Hilbert generic} we know that $M_2$ is spanned by the $2$-minors of $\Xgen$ and the quadratic part of the ideal $I$ has $\binom{mn+1}{2} - \binom{m}{2}\binom{n}{2} = \binom{m+1}{2}\binom{n+1}{2}$ quadratic generators. Such generators need to be orthogonal to the $2$-minors of $\Xgen$ with respect to the contraction pairing. One checks by direct inspection that the $\binom{m+1}{2}\binom{n+1}{2}$-many linearly indepedent polynomials in $\mathcal{G}$ satisfy this property. Thus, the quadratic part of $I$ is minimally generated by such (tidy) polynomials. We now want to prove that $I$ is actually generated in degree two, and the set $\mathcal{G}$ is in fact a revlex-universal Gr\"obner basis for $I$.

Fix any total order on the variables of $\Sgen$, let $\rev$ be the associated revlex order and consider the monomial ideal $U = (\init_{\rev}(g) \mid g \in \mathcal{G})$. To finish the proof, it is enough to show that the Hilbert functions of $\Sgen/U$ and $\Sgen/I$ coincide, i.e., by \Cref{lem:Hilbert generic}, that $\HF(\Sgen/U,s) = \binom{m}{s}\binom{n}{s}$ for every $0 \leq s \leq m$. Since $U$ can be written as $U' + (x_{ik}^2 \mid i \in [m],\ k \in [n])$, one has that the Hilbert function of $\Sgen/U$ coincides with the $f$-vector of the simplicial complex $\Delta$ whose Stanley--Reisner ideal is $U'$. Note that the complex $\Delta$ is flag, i.e., its Stanley--Reisner ideal is quadratically generated. We need to show that, for any fixed $0 \leq s \leq m$, the simplicial complex $\Delta$ has precisely $\binom{m}{s}\binom{n}{s}$ faces of cardinality $s$, i.e., of dimension $s-1$.

Since the monomials $x_{ik}x_{i\ell}$ and $x_{ik}x_{jk}$ are nonfaces of $\Delta$, it follows that in any $(s-1)$-dimensional face $x_{i_1,k_1}x_{i_2,k_2}\ldots x_{i_s,k_s}$ of $\Delta$ the indices $i_1, \ldots, i_s$ are all distinct, and likewise for $k_1, \ldots, k_s$. Now note that, if we pick two distinct indices $i,j$ in $[m]$ and two distinct indices $k,\ell$ in $[n]$, exactly one of $x_{ik}x_{j\ell}$ and $x_{i\ell}x_{jk}$ is an edge of $\Delta$: more precisely, the one containing the lowest variable among $x_{ik}, x_{i\ell}, x_{jk}, x_{j\ell}$.
    
We claim that choosing a set $\Ic$ of $s$ distinct indices in $[m]$ and a set $\Kc$ of $s$ distinct indices in $[n]$ gives rise to precisely one face of $\Delta$. This will prove that $\Delta$ has precisely $\binom{m}{s}\binom{n}{s}$ faces of dimension $s-1$, thus concluding the proof.

    Consider the monomial $\boldsymbol{m} = x_{i_1,k_1}x_{i_2,k_2}\ldots x_{i_s,k_s}$, where $\Ic = \{i_1,i_2, \ldots, i_s\}$ and $\Kc = \{k_1,k_2,\ldots, k_s\}$ both consist of $s$ distinct elements, and assume without loss of generality that $x_{i_1,k_1} < x_{i_2,k_2} < \ldots < x_{i_s,k_s}$. For any $r \in [s]$, let $\mathcal{V}_{\geq r}$ denote the set consisting of the variables $x_{ik}$ of $\Sgen$ with $i \in \{i_r,i_{r+1}, \ldots, i_s\}$ and $k \in \{k_r,k_{r+1}, \ldots, k_s\}$. 
    
    Assume there exists $r \in \{1, 2, \ldots, s-1\}$ such that $x_{i_r,k_r}$ is not the lowest variable $x_{j,\ell}$ in $\mathcal{V}_{\geq r}$. Since $x_{j,\ell} < x_{i_r, k_r} < x_{i_t, k_t}$ for any $t > r$, the indices $j \in \{i_r,i_{r+1}, \ldots, i_s\}$ and $\ell \in \{k_r,k_{r+1}, \ldots, k_s\}$ must appear in two distinct variables $x_{j,\ell'}$ and $x_{j',\ell}$ of $\boldsymbol{m}$, where $j' \in \{i_r,i_{r+1}, \ldots, i_s\} \setminus \{j\}$ and $\ell' \in \{k_r,k_{r+1}, \ldots, k_s\} \setminus \{\ell\}$. It follows from the previous discussion about edges of $\Delta$ that $x_{j,\ell'}x_{j',\ell}$ is a nonface of $\Delta$, and thus so is $\boldsymbol{m}$.

    Now note that there exists a unique way to order the elements of $\Ic$ and $\Kc$ so that the variable $x_{i_t,k_t}$ is the lowest in $\mathcal{V}_{\geq t}$ for every $t \in \{1, 2, \ldots, s\}$. When such orderings of $\Ic$ and $\Kc$ are chosen, the monomial $x_{i_a,k_a}x_{i_b,k_b}$ is an edge of $\Delta$ for any possible choice of distinct $a, b$ in $\{1, 2, \ldots, s\}$. Since $\Delta$ is a flag simplicial complex, this implies that $\boldsymbol{m}$ itself is a face of $\Delta$, thus concluding the proof.
\end{proof}

\begin{proposition} \label{prop:generalized permanents}
    The generalized $2$-permanents of the generic $(m \times n)$-matrix $\Xgen$ define a strongly Koszul algebra with respect to the variables.
\end{proposition}
\begin{proof}
If $\mathrm{char}(\field) \neq 2$, the generalized $2$-permanents of $\Xgen$ coincide (up to a nonzero multiplicative constant) with the polynomials in $\mathcal{G}$ from \Cref{prop:generic apolarity}, and thus the claim follows directly from \Cref{prop:generic apolarity}.

If $\mathrm{char}(\field) = 2$, the set of (nonzero) generalized $2$-permanents coincides with the set of $2$-minors of $\Xgen$, and thus the claim follows from \Cref{thm:2-minors}(i), taking as $G$ the edgeless bipartite graph.
\end{proof}

\begin{remark}
One can start from the maximal \emph{permanents} of $\Xgen$ and obtain a result very similar to \Cref{prop:generic apolarity} (of course, if $\mathrm{char}(\field)=2$, one gets literally the same result). Indeed, since permanents also satisfy the Laplace expansion rule, for every $0 \leq s \leq m$ one has that the $s$-th graded part of the $\Sgen$-module (with respect to the contraction action) generated by maximal permanents is minimally generated as a $\field$-vector space by the $s$-permanents of $\Xgen$, and thus $\HF(\Sgen/I, s) = \binom{m}{s}\binom{n}{s}$. Reasoning as in the determinantal case, one gets that the quadratic part of $I$ is minimally generated by the following $\binom{m+1}{2}\binom{n+1}{2}$ quadrics:
\begin{itemize}
    \item {\makebox[3cm]{$x_{ik}^2$\hfill}} for $i \in [m],\ k \in [n]$;
    \item {\makebox[3cm]{$x_{ik}x_{jk}$\hfill}} for $i,j \in [m], \ k \in [n], \ i \neq j$;
    \item {\makebox[3cm]{$x_{ik}x_{i\ell}$\hfill}} for $i \in [m],\ k,\ell \in [n],\ k \neq \ell$;
    \item {\makebox[3cm]{$x_{ik}x_{j\ell}-x_{i\ell}x_{jk}$\hfill}} for $i,j \in [m], \ k,\ell \in [n],\ i \neq j, \ k \neq \ell$
    \end{itemize}
When taking the ideal generated by the initial forms with respect to a given revlex monomial order $\rev$, we get back precisely what we called $U$ in the determinantal case, and we have already shown that $\HF(\Sgen/U,s) = \binom{m}{s}\binom{n}{s}$ for every $0 \leq s \leq m$. This proves that the polynomials listed above form a tidy revlex-universal Gr\"obner basis of quadrics for $I$, and hence the ideal $I$ defines a strongly Koszul algebra with respect to the variables. Like in the previous determinantal case, Shafiei studied the case when $\Xgen$ is square, determining the generators of $I$ and stating they form a quadratic Gr\"obner basis \cite[Theorem 2.13, Corollary 5.4]{Sha}.
\end{remark}

\subsection{Maximal Pfaffians and apolarity}

\begin{notation}
Let $Y$ be any skew-symmetric matrix of size $m$. In this section, if $\Hc \subseteq [m]$, we will denote by $Y_{\Hc}$ the skew-symmetric matrix obtained from $Y$ by deleting all the rows and columns indexed by elements of $\Hc$.
\end{notation}

If $Z$ is a skew-symmetric matrix of even size $2m'$ and $1 \leq j \leq 2m'$, it is known that the following useful ``Laplace-like'' formula holds (see, e.g., \cite[Equation D.1]{FultonPragacz}):

\begin{equation} \label{eq:Pfaffian expansion}
\pf(Z) = \sum_{i<j}(-1)^{i+j-1}\cdot z_{ij} \cdot \pf(Z_{\{i,j\}}) + \sum_{i>j}(-1)^{i+j}\cdot z_{ij} \cdot \pf(Z_{\{i,j\}}),
\end{equation}
where $\pf(\text{--})$ denotes the maximal Pfaffian of a skew-symmetric matrix of even size.

Let $\Xsk$ be a generic $N \times N$ skew-symmetric matrix with entries in the polynomial ring $\Ssk = \field[x_{ij} \mid 1 \leq i < j \leq N]$. Let $n = \left\lfloor\frac{N}{2}\right\rfloor$ and consider the $\Ssk$-module $M$ (with respect to the contraction action) finitely generated by the $2n$-Pfaffians of $X$. Note that such Pfaffians are $\field$-linear combinations of squarefree monomials, as can be derived for instance from \eqref{eq:Pfaffian expansion}. The ideal $I = \mathrm{ann}(M)$ defines a Gorenstein ring when $N$ is even and a level non-Gorenstein ring when $N$ is odd; its generators were found by Shafiei in the case when $N$ is even \cite[Theorem 4.12]{Sha}. It is our next goal to write down explicitly the generators of $I$ for any $N$ and to prove that they define a strongly Koszul algebra. To do so, let us first determine the Hilbert function of $\Ssk/I$. Thanks to the properties of Macaulay's inverse system and \Cref{rem:differentiation and contraction for squarefree monomials}, this can be gauged by analyzing the partial derivatives of the $2n$-Pfaffians of $\Xsk$. 

\begin{lemma} \label{lem:Hilbert Pfaffian}
    For every $0 \leq s \leq n$, one has that $M_s = \left\langle 2s\text{-Pfaffians of }\Xsk\right\rangle$ and, thus, that $\HF(\Ssk/I, s) = \binom{N}{2s}$.
\end{lemma}
\begin{proof}
    Fix $0 \leq s \leq n$, let $i_1, \ldots, i_{n-s}, j_1, \ldots, j_{n-s}$ be elements of $[N]$ such that $i_k < j_k$ for every $k \in [n-s]$, and denote by $\Hc$ the set $\{i_1, \ldots, i_{n-s}, j_1, \ldots, j_{n-s}\}$. Let us first consider the case when $N = 2n$. It is a direct consequence of \eqref{eq:Pfaffian expansion} that 
    \[
    \frac{\partial^{n-s} (\pf(\Xsk))}{\partial x_{i_1, j_1} \partial x_{i_2, j_2} \ldots \partial x_{i_{n-s}, j_{n-s}}} = \begin{cases}0 & |\Hc|<2(n-s) \\ \pm \pf(\Xsk_{\Hc}) & |\Hc|=2(n-s),\end{cases}
    \]
    and thus $M_s$ is generated as a $\field$-vector space by the $2s$-Pfaffians of $\Xsk$.

    The case when $N = 2n+1$ is similar, but this time \[M_s = \left\langle \partial^{\boldsymbol{\alpha}}\pf(\Xsk_{\{\ell\}}) \mid \ell \in [2n+1], \ |\boldsymbol{\alpha}|=n-s \right\rangle.\]

    It follows again from \eqref{eq:Pfaffian expansion} that 
    \[
    \frac{\partial^{n-s} (\pf(\Xsk_{\{\ell\}}))}{\partial x_{i_1, j_1} \partial x_{i_2, j_2} \ldots \partial x_{i_{n-s}, j_{n-s}}} = \begin{cases}0 & |\Hc \cup \{\ell\}|<2(n-s)+1 \\ \pm \pf(\Xsk_{\Hc \cup \{\ell\}}) & |\Hc \cup \{\ell\}|=2(n-s)+1,\end{cases}
    \]
    and thus $M_s$ is generated as a $\field$-vector space by the $2s$-Pfaffians of $\Xsk$.
    
    In both cases, since the collection of $2s$-Pfaffians of $\Xsk$ is a $\field$-linearly independent set, it follows that \[\HF(\Ssk/I, s) = \dim_{\field}M_s = \binom{N}{2s}.\]
\end{proof}

\begin{proposition} \label{prop:Pfaffian apolarity}
    Let $M$ be the $\Ssk$-module (with respect to the contraction action) generated by the maximal even Pfaffians of $\Xsk$ and let $I = \mathrm{ann}(M)$. Then \[\begin{split}\mathcal{G} = \!\ &\{x_{ij}^2 \mid 1 \leq i < j \leq N\} \cup \{x_{ij}x_{ik},\ x_{ij}x_{jk}, \ x_{ik}x_{jk} \mid 1 \leq i < j < k \leq N\}\\ \cup\ \! &\{x_{ij}x_{k\ell}+x_{ik}x_{j\ell},\ x_{i\ell}x_{jk}+x_{ik}x_{j\ell},\ x_{ij}x_{k\ell} - x_{i\ell}x_{jk} \mid 1 \leq i < j < k < \ell \leq N\} \end{split}\] is a tidy revlex-universal Gr\"obner basis of $\binom{N}{2} + 3\binom{N}{3} + 3\binom{N}{4}$ quadrics for $I$. Removing one polynomial from each of the sets $\{x_{ij}x_{k\ell}+x_{ik}x_{j\ell},\ x_{i\ell}x_{jk}+x_{ik}x_{j\ell},\ x_{ij}x_{k\ell} - x_{i\ell}x_{jk}\}$ yields a minimal generating set for $I$.
\end{proposition}

\begin{corollary} \label{cor:Pfaffian apolarity}
    Let $M$ be the $\Ssk$-module (with respect to the contraction action) generated by the maximal even Pfaffians of $\Xsk$. Then the ideal $I = \mathrm{ann}(M)$ defines a strongly Koszul algebra with respect to the variables.
\end{corollary}

\begin{proof}[Proof of \Cref{prop:Pfaffian apolarity}]
    By \Cref{lem:Hilbert Pfaffian} we know that $M_2$ is spanned by the $4$-Pfaffians of $\Xsk$ and the quadratic part of the ideal $I$ has $\binom{\binom{N}{2}+1}{2} - \binom{N}{4} = \binom{N}{2} + 3\binom{N}{3} + 2\binom{N}{4}$ quadratic generators. Such generators need to be orthogonal to the $4$-Pfaffians of $\Xsk$ with respect to the contraction pairing. Since every choice of $1 \leq i < j < k <\ell \leq N$ yields a $4$-Pfaffian $x_{ij}x_{k\ell} - x_{ik}x_{j\ell} + x_{i\ell}x_{jk}$, it is readily seen that the quadratic part of $I$ contains all the polynomials in $\mathcal{G}$, i.e.,
    \begin{enumerate}
    \item {\makebox[8.5cm]{$x_{ij}^2$\hfill}} for $1 \leq i < j \leq N$;
    \item {\makebox[8.5cm]{$x_{ij}x_{ik}, \ x_{ij}x_{jk}, \ x_{ik}x_{jk}$\hfill}} for $1 \leq i < j < k \leq N$;
    \item {\makebox[8.5cm]{$x_{ij}x_{k\ell}+x_{ik}x_{j\ell},\ x_{i\ell}x_{jk}+x_{ik}x_{j\ell},\ x_{ij}x_{k\ell} - x_{i\ell}x_{jk}$\hfill}} for $1 \leq i < j < k < \ell \leq N$.
    \end{enumerate}

    The above polynomials are $\binom{N}{2} + 3\binom{N}{3} + 3\binom{N}{4}$ many; to generate minimally the quadratic part of $I$, it is enough to discard one of the three linearly dependent polynomials of type (iii) for every quadruple of indices. However, this redundancy makes $\mathcal{G}$ into a revlex-universal Gr\"obner basis of quadrics for $I$, as we now explain.
    
    Consider any total order on the variables of $\Ssk$, let $\rev$ be the associated revlex order and consider the monomial ideal $U = (\init_{\rev}(g) \mid g \in \mathcal{G})$. To finish the proof, it is enough to prove that the Hilbert function of $\Ssk/U$ and the Hilbert function of $\Ssk/I$ coincide, i.e., by \Cref{lem:Hilbert Pfaffian}, that $\HF(\Ssk/U,s) = \binom{N}{2s}$ for every $0 \leq s \leq n$. 
    Since $U$ can be written as $U' + (x_{ij}^2 \mid 1 \leq i < j \leq N)$, the Hilbert function of $\Ssk/U$ coincides with the $f$-vector of the (flag) simplicial complex $\Delta$ whose Stanley--Reisner ideal is $U'$. 
    We hence need to show that, for any fixed $0 \leq s \leq n$, the simplicial complex $\Delta$ has precisely $\binom{N}{2s}$ faces of cardinality $s$, i.e., of dimension $s-1$. Due to the presence of $x_{ij}x_{ik}$, $x_{ij}x_{jk}$ and $ x_{ik}x_{jk}$ inside the Stanley--Reisner ideal $U'$, any $(s-1)$-dimensional face $x_{i_1,i_2}x_{i_3,i_4}\ldots x_{i_{2s-1},i_{2s}}$ of $\Delta$ must feature $2s$ distinct indices.
    
    For the rest of this proof, whenever $a$ and $b$ are distinct numbers in $[N]$, we will write either $x_{ab}$ or $x_{ba}$ to denote the variable $x_{\min\{a,b\}, \max\{a,b\}}$ of $\Ssk$.

    We claim that there is a unique way to partition any four distinct indices $i,j,k,\ell$ in $[N]$ into two pairs so that the associated variables form an edge of $\Delta$; this happens if and only if one of the two pairs corresponds to the variable of $\Ssk$ that is lowest among those featuring $i,j,k,\ell$ as their indices. Indeed, assume without loss of generality that $x_{ij}$ has this property. Then the revlex order $\rev$ selects $x_{ik}x_{j\ell}$ and $x_{i\ell}x_{jk}$ as the leading terms of respectively $x_{ik}x_{j\ell} \pm x_{ij}x_{k\ell}$ and $x_{i\ell}x_{jk} \pm x_{ij}x_{k\ell}$, and thus $x_{ik}x_{j\ell}$ and $x_{i\ell}x_{jk}$ are nonfaces of $\Delta$. All other monomials in $U'$ do not affect $x_{ij}x_{k\ell}$, that is hence an edge of $\Delta$.
    
    We claim that choosing any set $\Hc$ of $2s$ distinct indices in $[N]$ gives rise to precisely one face of $\Delta$. This will prove that $\Delta$ has precisely $\binom{N}{2s}$ faces of dimension $s-1$, as desired. 

    Consider the monomial $\boldsymbol{m} = x_{i_1,i_2}x_{i_3,i_4}\ldots x_{i_{2s-1},i_{2s}}$, where $\Hc = \{i_1,i_2,\ldots, i_{2s}\}$ consists of $2s$ distinct elements, and assume without loss of generality that $x_{i_1,i_2} < x_{i_3,i_4} < \ldots < x_{i_{2s-1},i_{2s}}$. For any $r \in [s]$, let $\mathcal{V}_{\geq {2r-1}}$ denote the set consisting of the variables of $\Ssk$ whose indices belong to $\{i_{2r-1},i_{2r}, \ldots, i_{2s-1},i_{2s}\}$.
    
    Assume there exists $r \in \{1, 2, \ldots, s-1\}$ such that $x_{i_{2r-1},i_{2r}}$ is not the lowest variable $x_{h,k}$ in $\mathcal{V}_{\geq {2r-1}}$. Since $x_{h,k} < x_{i_{2r-1}, i_{2r}} < x_{i_{2t-1}, i_{2t}}$ for any $t > r$, the distinct elements $h$ and $k$ in $\{i_{2r-1},i_{2r}, \ldots, i_{2s-1},i_{2s}\}$ must appear in two distinct variables $x_{hk'}$ and $x_{h'k}$ of $\boldsymbol{m}$, where $h'$ and $k'$ are distinct elements in $\{i_{2r-1},i_{2r}, \ldots, i_{2s-1},i_{2s}\} \setminus \{h,k\}$. It follows from the previous discussion about edges of $\Delta$ that $x_{hk'}x_{h'k}$ is a nonface of $\Delta$, and thus so is $\boldsymbol{m}$.

    Now note that there exists a unique way to order the elements of $\Hc$ so that, for every $t \in \{1, 2, \ldots, s\}$, one has that $i_{2t-1} < i_{2t}$ and the variable $x_{i_{2t-1},i_{2t}}$ is the lowest in $\mathcal{V}_{\geq {2t-1}}$. Under such an ordering of $\Hc$, the monomial $x_{i_{2a-1},i_{2a}}x_{i_{2b-1},i_{2b}}$ is an edge of $\Delta$ for any possible choice of distinct $a, b$ in $\{1, 2, \ldots, s\}$. Since $\Delta$ is a flag simplicial complex, this implies that $\boldsymbol{m}$ itself is a face of $\Delta$, thus concluding the proof.
\end{proof}

\subsection{A connection to Severi varieties} \label{subsec:Severi}

For this section, let $\field=\mathbb{C}$. It was proved by Zak \cite{Zak} that, if $X$ is an $n$-dimensional nondegenerate (i.e.,~not contained in any hyperplane) projective variety in $\mathbb{P}^{\frac{3}{2}n+2}$ whose secant variety of lines does not fill the ambient space, then $X$ is (smooth and) projectively equivalent to one of the four \emph{Severi varieties}, i.e.:
\begin{enumerate}
\item the $2$-dimensional Veronese variety $\nu_2(\mathbb{P}^2) \subseteq \mathbb{P}^5$;
\item the $4$-dimensional Segre variety $\mathbb{P}^2 \times \mathbb{P}^2 \subseteq \mathbb{P}^8$;
\item the $8$-dimensional Grassmannian of $2$-dimensional subspaces of $\mathbb{C}^6$ (or, equivalently, of lines in $\mathbb{P}^5$), i.e., $\mathrm{Gr}(2,6) \subseteq \mathbb{P}^{14}$;
\item the $16$-dimensional octonionic plane (also known as Cayley plane or $E_6$-variety), i.e., $\mathbb{OP}^2 \subseteq \mathbb{P}^{26}$.
\end{enumerate}
Moreover, in all of the above cases, the secant variety of lines is a cubic hypersurface. We wish to prove the following result:

\begin{theorem} \label{thm:Severi}
    Let $X \subseteq \mathbb{P}^N$ be a Severi variety, let $S$ be the coordinate ring of $\mathbb{P}^N$ and let $F$ be the cubic polynomial defining the secant variety of lines of $X$. Then the Artinian Gorenstein ring $S/\mathrm{ann}(F)$ is strongly Koszul with respect to the variables.
\end{theorem}

To prove \Cref{thm:Severi} for the octonionic plane, we need to introduce some of its properties. Let $X = \mathbb{OP}^2$ be the octonionic plane embedded in $\mathbb{P}^{26}$. Then the secant variety $SX$ is defined by a cubic hypersurface that is invariant under the action of the group $E_6$. The defining polynomial $F$ of this hypersurface (sometimes known as the Cayley cubic hypersurface) can be expressed as the determinant of a generic $3 \times 3$ $\mathbb{O}$-Hermitian matrix \cite[Section 3.1]{IlievManivel}, but such an expression is not optimal for our purposes. Instead, we will use a result of Lurie \cite{Lurie} connecting the Cayley cubic hypersurface to the geometry of the 27 lines on a smooth cubic surface in $\mathbb{P}^3$. We briefly recall some facts about the latter, using \cite[Section 9.1]{Dolgachev} as our main source. Any smooth cubic surface in $\mathbb{P}^3$ arises as the blowup of six points $p_1, \ldots, p_6$ of $\mathbb{P}^2$ in general position, and the 27 lines on the surface can be subdivided into ``$a$-lines, $b$-lines and $c$-lines'' (this is not standard notation!) as follows:
\begin{itemize}
    \item for each $i \in [6]$, one has a line $a_i$ corresponding to the exceptional divisor of the blowup at $p_i$;
    \item for each $i \in [6]$, there is a line $b_i$ arising from the strict transform of the conic of $\mathbb{P}^2$ passing through all the $p_j$'s except $p_i$;
    \item for every $i,j \in [6]$ with $i<j$, one has a line $c_{ij}$ corresponding to the strict transform of the line through $p_i$ and $p_j$. (In what follows, the symbol $c_{ji}$ will also denote $c_{ij}$.)
\end{itemize}

Moreover, there exist 45 \emph{tritangent planes}, i.e., planes of $\mathbb{P}^3$ that intersect the surface in exactly three of the 27 lines and are tangent to the surface in the (possibly coincident) points where the three lines intersect pairwise. Every tritangent plane can be identified by enumerating which lines on the surface it contains; in what follows, we will write $\ell mn$ (and possibly permutations thereof) to denote the tritangent plane containing the lines $\ell$, $m$ and $n$. The 45 tritangent planes can be listed as follows (compare, e.g., \cite[Lemma 9.1.8]{Dolgachev}):
\begin{itemize}
    \item for every $i,j \in [6]$ with $i<j$, one has the two planes $a_ib_jc_{ij}$ and $a_jb_ic_{ij}$;
    \item for every partition of $[6]$ into three pairs $(i,j), (k,\ell), (m,n)$, one has the plane $c_{ij}c_{k\ell}c_{mn}$.
\end{itemize}

In particular, each of the 27 lines on the cubic surface is contained in exactly five distinct tritangent planes. We record in the following lemma some observations that will be useful in the proof of \Cref{thm:Severi}.

\begin{lemma} \label{lem:27 lines}
    Let $\mathcal{S}$ be a smooth cubic surface in $\mathbb{P}^3$ and let $\mathcal{L}$ be the set of the 27 lines contained in $\mathcal{S}$. Then:
    \begin{enumerate}
        \item Given any $\mathcal{L}' \subseteq \mathcal{L}$ with $|\mathcal{L}'| \geq 4$, there exist $\ell_1$, $\ell_2$ in $\mathcal{L}'$ such that $\ell_1 \cup \ell_2$ is not contained in any tritangent plane.
        \item Let $\ell_1$, $\ell_2$ and $\ell_3$ be three distinct lines in $\mathcal{L}$ and suppose that for every choice of two distinct indices $i,j \in [3]$ there exists a tritangent plane $\pi_{\ell_i,\ell_j}$ containing the union of $\ell_i$ and $\ell_j$. Then $\pi_{\ell_1, \ell_2} = \pi_{\ell_1, \ell_3} = \pi_{\ell_2,\ell_3} = \ell_1\ell_2\ell_3$. 
        \item For any $\ell \in \mathcal{L}$ and any tritangent plane $\rho$ not containing $\ell$, there exists a tritangent plane $\sigma$ that contains $\ell$ and meets $\rho$ in another line $\ell' \in \mathcal{L}$. 
    \end{enumerate}
\end{lemma}
\begin{proof}
\phantom{.}
    \begin{enumerate}
        \item A direct inspection of the list of tritangent planes yields immediately that, if $\mathcal{L}'$ contains two distinct $a$-lines or two distinct $b$-lines, then the claim holds. Assume otherwise. Then $\mathcal{L}'$ contains at least two $c$-lines. If at least two of them share an index, the claim holds. If instead any two $c$-lines in $\mathcal{L'}$ do not share any indices, then $\mathcal{L}'$ contains at most three $c$-lines, and hence it must contain at least one $a$- or $b$-line as well, say $a_i$ without loss of generality. Then at least one of the $c$-lines in $\mathcal{L}'$ does not feature the index $i$, and the claim holds. 
        \item Let $\ell_1$, $\ell_2$, $\ell_3$ be distinct lines in $\mathcal{L}$ such that the union of any two of them is contained in some tritangent plane. Then at least one of $\ell_1$, $\ell_2$ and $\ell_3$ must be a $c$-line; without loss of generality, say $\ell_1 = c_{ij}$ for some distinct indices $i$ and $j$. If one of $\ell_2$ and $\ell_3$ is also a $c$-line, then it must be that $\ell_2 = c_{kp}$ and $\ell_3 = c_{qr}$, with $\{i,j,k,p,q,r\} = [6]$. Otherwise, it must be that $\ell_2 = a_i$ and $\ell_3 = b_j$ or $\ell_2 = a_j$ and $\ell_3 = b_i$. In all cases, the claim holds.
        \item We have to go through several cases according to which line and plane we select. Since the cases with $\ell = b_i$ are entirely analogous to the ones with $\ell = a_i$, we will list only the latter.
        \begin{itemize}
        \item If $\ell = a_i$ and $\rho = a_hb_kc_{hk}$ (with $h \neq i$), one can pick $\sigma = a_ib_kc_{ik}$ (if $k \neq i$) or $\sigma = a_ib_hc_{hi}$ (if $k=i$).
        \item If $\ell = a_i$ and $\rho = c_{ij}c_{pq}c_{rs}$, one can pick $\sigma = a_ib_jc_{ij}$.
        \item If $\ell = c_{ij}$ and $\rho = a_hb_kc_{hk}$ (with $\{h,k\} \neq \{i,j\})$, there are two subcases: 
        \begin{itemize}
           \item if $\{i,j\} \cap \{h,k\} = \varnothing$, one can pick $\sigma = c_{ij}c_{hk}c_{pq}$ with $\{i,j,h,k,p,q\} = [6]$;
           \item if $|\{i,j\} \cap \{h,k\}| = 1$, then we can assume without loss of generality that $i \in \{h,k\}$. One can then pick $\sigma = a_ib_jc_{ij}$ (if $h=i$) or $\sigma = a_jb_ic_{ij}$ (if $k=i$).
        \end{itemize}
        \item If $\ell = c_{ij}$ and $\rho = c_{ih}c_{jk}c_{pq}$ with $\{i,j,h,k,p,q\} = [6]$, one can pick $\sigma = c_{ij}c_{hk}c_{pq}$. 
        \end{itemize}
    \end{enumerate}
\end{proof}

\begin{proof}[Proof of \Cref{thm:Severi}]
We analyze separately the cases (i)--(iv) from Zak's classification of Severi varieties.

\begin{enumerate}
\item When $X$ is the Veronese surface with its standard embedding in $\mathbb{P}^5$, the secant variety $SX$ is defined by the determinant $F$ of the $3 \times 3$ generic symmetric matrix. One checks by direct computation that $S/\mathrm{ann}(F)$ yields a strongly Koszul algebra with respect to the variables (but see \Cref{qu:Veronese} and the discussion preceding it).
\item When $X$ is the Segre embedding of $\mathbb{P}^2 \times \mathbb{P}^2$ into $\mathbb{P}^8$, the secant variety $SX$ is defined by the determinant $F$ of the $3 \times 3$ generic matrix. Then, by \Cref{prop:generic apolarity}, $\mathrm{ann}(F)$ admits a tidy revlex-universal Gr\"obner basis of quadrics, and hence $S/\mathrm{ann}(F)$ is strongly Koszul with respect to the variables.
\item When $X = \mathrm{Gr}(2,6)$ with its standard Pl\"ucker embedding in $\mathbb{P}^{14}$, the secant variety $SX$ is defined by the $6$-Pfaffian of the generic $6 \times 6$ skew-symmetric matrix. Then, by \Cref{prop:Pfaffian apolarity}, $\mathrm{ann}(F)$ admits a tidy revlex-universal Gr\"obner basis of quadrics, and hence $S/\mathrm{ann}(F)$ is strongly Koszul with respect to the variables.
\item Let $X = \mathbb{OP}^2$ be the octonionic plane embedded in $\mathbb{P}^{26}$ and let $F$ be the cubic polynomial defining its secant variety $SX$. Name the 27 variables of the polynomial ring $S$ so that they match the $a, b, c$-notation introduced previously for lines on the cubic surface. Lurie proved in \cite{Lurie} (see also \cite[Remark 9.1.15]{Dolgachev}, \cite[Section 2.2]{BBF}) that $F$ can be written as a signed combination of the 45 squarefree cubic monomials of $S$ corresponding to tritangent planes (here, as previously discussed, we identify any given tritangent plane via the lines on the surface contained in it). The precise sign pattern will not matter to us; in fact, all we need is that $F$ is a linear combination of the tritangent planes, and every tritangent plane appears with a nonzero coefficient in such a combination. 
Now consider the apolar ideal $I = \mathrm{ann}(F)$. By \cite[\S 2.2]{apolarity}, the ideal $I$ contains no linear forms if and only if the projective variety $V(F)$ is not a cone; in our case, $V(F)=SX$, and Zak proved in \cite[Lemma 2]{Zak} that $SX$ is not a cone. Hence, the ideal $I$ contains no linear forms.
Since $F$ is a cubic polynomial in 27 variables and $I$ contains no linear forms, the Hilbert series of the Artinian Gorenstein ring $S/I$ must be $1+27t+27t^2+t^3$. Moreover, $I$ contains the square of every variable of $S$ and every monomial $\ell\ell'$ where $\ell$ and $\ell'$ are two distinct lines on the cubic surface whose union is not contained in any tritangent plane. To prove that $S/I$ is strongly Koszul with respect to the variables, it is enough to show that the following set $\mathcal{G}$ of $3 \cdot (6 + \binom{6}{2} + 6 \cdot \binom{5}{2}) = 243$ quadratic monomials and $27 \cdot \binom{5}{2} = 270$ quadratic binomials in $I$ form a tidy revlex-universal Gr\"obner basis for $I$:

\begin{itemize}
\item $a_i^2$, $b_i^2$, $a_ib_i$, $c_{ij}^2$, $a_ia_j$, $b_ib_j$, $a_ic_{jk}$, $b_ic_{jk}$, $c_{ij}c_{ik}$, where $i, j, k$ are distinct indices in $[6]$;
\item $\frac{\pi}{\ell} \pm \frac{\pi'}{\ell}$, where $\ell$ is any of the 27 lines and $\pi$, $\pi'$ are any two of the five tritangent planes containing $\ell$. (The exact sign depends on the sign pattern in $F$, but we won't need to be more precise here.)
\end{itemize}

Fix a total order on the variables of $S$ (i.e., on the 27 lines) and let $\rev$ be the associated revlex order. We want to prove that $\mathcal{G}$ is a (tidy) Gr\"obner basis for $I$ with respect to $\rev$. Let $J$ be the monomial ideal of $S$ generated by $\{\init_{\rev}(g) \mid g \in \mathcal{G}\}$. Then $J \subseteq \init_{\rev}(I)$, and to prove our claim it is enough to show that the Hilbert functions of $S/J$ and $S/I$ coincide. Since $\HF(S/J, i) \geq \HF(S/I,i)$ for every $i \in \NN$, it is enough to prove the following three statements:
\smallskip
\begin{enumerate}[label=(\Alph*)]
    \item The ideal $J$ contains all monomials of $S$ of degree $4$ and higher.
    \item The ideal $J$ contains at least $\binom{28}{2}-27 = 351$ quadratic monomials.
    \item The ideal $J$ contains every cubic monomial of $S$ different from the one corresponding to the revlex-lowest tritangent plane.
\end{enumerate}
\smallskip

Let us prove claim (A) first. Let $m$ be a monomial of $S$ of degree $4$ or higher. If $m$ is not squarefree, then it is divided by $\ell^2 \in J$ for some line $\ell$. If instead $m$ is squarefree then, by \Cref{lem:27 lines}(i), there exists a quadratic monomial of $J$ dividing $m$, and thus we are done.

Next, let us consider claim (B). For every line $\ell$ contained in the cubic surface, there exist exactly five tritangent planes containing $\ell$. Let $\pi$ be the revlex-lowest among them. Then, for every other tritangent plane $\pi'$ containing $\ell$, one has that $\frac{\pi'}{\ell} \pm \frac{\pi}{\ell} \in \mathcal{G}$, and $\init_{\rev}(\frac{\pi'}{\ell} \pm \frac{\pi}{\ell}) = \frac{\pi'}{\ell} \in J$. Since no two distinct tritangent planes can share two of the $27$ lines, we have just found $27 \cdot (5-1) = 108$ quadratic monomials in $J$. Putting together these monomials with the $243$ already present in $\mathcal{G}$, we get that $J$ contains at least $351$ quadratic monomials, as desired.

Finally, let us prove claim (C). Let $\ell_{\min}$ be the lowest line according to the total order we fixed and let $\pi_{\min}$ denote either the revlex-lowest tritangent plane or the squarefree cubic monomial in $S$ associated with it. Note that the plane $\pi_{\min}$ must contain $\ell_{\min}$. Let $m$ be any cubic monomial in $S$ different from $\pi_{\min}$. If $m$ is not squarefree, then it is a multiple of the square of some variable, and such a quadratic monomial lies in $J$. Assume now that $m$ is squarefree, and let $m = L_1L_2L_3$. If there exist two distinct indices $i,j \in [3]$ such that $L_i \cup L_j$ is not contained in any tritangent plane, then the monomial $L_iL_j$ lies in $J$ and divides $m$. Assume instead that $L_1 \cup L_2$, $L_1 \cup L_3$ and $L_2 \cup L_3$ are all contained in some tritangent plane. Then, by \Cref{lem:27 lines}(ii), the lines $L_1$, $L_2$ and $L_3$ are all contained in the same tritangent plane $\rho$, and $\rho = L_1L_2L_3 = m \neq \pi_{\min}$.
\begin{itemize}
    \item If $\rho$ contains $\ell_{\min}$, then $\frac{\rho}{\ell_{\min}} \pm \frac{\pi_{\min}}{\ell_{\min}} \in \mathcal{G}$ and $\init_{\rev}(\frac{\rho}{\ell_{\min}} \pm \frac{\pi_{\min}}{\ell_{\min}}) = \frac{\rho}{\ell_{\min}}$, since otherwise $\rho$ would be revlex-lower than $\pi_{\min}$. Hence, $\frac{\rho}{\ell_{\min}}$ lies in $J$ and so does $\rho$.
    \item If $\rho$ does not contain $\ell_{\min}$, then by \Cref{lem:27 lines}(iii) there exists a tritangent plane $\sigma$ containing $\ell_{\min}$ and meeting $\rho$ in another line $L$ on the cubic surface. Then $\frac{\rho}{L} \pm \frac{\sigma}{L} \in \mathcal{G}$. Moreover, since $\ell_{\min}$ divides $\frac{\sigma}{L}$ but not $\frac{\rho}{L}$, one has that $\frac{\sigma}{L}$ is revlex-lower than $\frac{\rho}{L}$, and so $\init_{\rev}(\frac{\rho}{L} \pm \frac{\sigma}{L}) = \frac{\rho}{L}$. Thus, $\frac{\rho}{L}$ lies in $J$ and so does $\rho$. 
\end{itemize}
\end{enumerate}
\end{proof}

The previous proof relies on algebraic-combinatorial methods: it is not a priori clear how strong Koszulness interacts with the rich geometric information available in the setting of \Cref{thm:Severi}. It is thus natural to ask:

\begin{question}
    Is there a geometric interpretation of \Cref{thm:Severi}?
\end{question}

Going back to algebra, the ideal $\mathrm{ann}(F)$ turns out to admit a tidy revlex-universal Gr\"obner basis of quadrics, at least after a suitable linear change of coordinates, in cases (ii), (iii) and (iv) of \Cref{thm:Severi}. It is unclear to us whether this holds also for case (i), i.e., when $X$ is the Veronese surface in $\mathbb{P}^5$ (even though the universal Gr\"obner basis in the usual coordinates turns out to be tidy, thus simplifying the computation of colon ideals). If the homogeneous ideal $I \subseteq S$ admits a quadratic Gr\"obner basis and the quotient ring $S/I$ is Artinian, a result by Eisenbud, Reeves and Totaro \cite[Theorem 19]{ERT} implies that $I$ must contain the square of the lowest variable in $S$. It follows that, if $F \in S_d$ and $\mathrm{ann}(F)$ admits a revlex-universal Gr\"obner basis of quadrics, then $\mathrm{ann}(F)$ must contain the square of \emph{every} variable in $S$; in particular, $F$ must be a linear combination of squarefree monomials. This prompts the following concrete questions:

\begin{question} \label{qu:Veronese}
Let $S = \mathbb{C}[x_{i,j} \mid 1 \leq i \leq j \leq 3]$ and let $F$ be the determinant of the generic symmetric $3 \times 3$ matrix with entries in $S$, i.e., \[F = x_{1,1}x_{2,2}x_{3,3} + 2x_{1,2}x_{1,3}x_{2,3} - (x_{1,2}^2x_{3,3} + x_{1,3}^2x_{2,2} + x_{2,3}^2x_{1,1}).\] 

\begin{enumerate}
    \item Does there exist a linear change of coordinates $g \in \mathrm{GL}(S_1)$ such that $\mathrm{ann}(gF)$ admits a (tidy) revlex-universal Gr\"obner basis of quadrics?
    \item The linear change of coordinates \[g'\colon\begin{cases}x_{1,1} \mapsto x_{1,1}-x_{1,2}-x_{1,3}\\ x_{1,2} \mapsto x_{1,2}\\ x_{1,3} \mapsto x_{1,3} \\ x_{2,2} \mapsto -x_{1,2}+x_{2,2}-x_{2,3} \\ x_{2,3} \mapsto x_{2,3} \\ x_{3,3} \mapsto -x_{1,3}-x_{2,3}+x_{3,3}\end{cases}\]
    transforms $F$ into a signed sum of $16$ squarefree cubic monomials. (However, $\mathrm{ann}(g'F)$ does not admit a revlex-universal Gr\"obner basis of quadrics, since $S/\mathrm{ann}(g'F)$ is not strongly Koszul with respect to the variables.) What is the minimum $N$ such that $gF$ is a linear combination of $N$ squarefree cubic monomials for some $g \in \mathrm{GL}(S_1)$?
\end{enumerate}
\end{question}

\section{Strongly Koszul algebras are not always G-quadratic} \label{sec:not G-quadratic}
After learning about the results in \Cref{sec:tidy}, the reader might suspect that the defining ideal of a ring that is strongly Koszul with respect to the variables must admit a quadratic Gr\"obner basis, maybe even a revlex one. This was asked already in the original paper \cite{HHR} by Herzog, Hibi and Restuccia (for the toric case) and in a later survey \cite{CDR} by Conca, De Negri and Rossi (for the general case).

\begin{question} \label{qu:G-quadratic}
\phantom{.}
    \begin{enumerate}
    \item \cite[Question 4.13(1)]{CDR} If $R=S/I$ is strongly Koszul, does $I$ admit a Gr\"obner basis of quadrics (possibly after a change of coordinates)?
    \item \cite[p.~166]{HHR} If $R=S/I$ is a strongly Koszul toric ring, does $I$ admit a Gr\"obner basis of quadrics?
    \end{enumerate}
\end{question}

The aim of this section is to provide a negative answer to \Cref{qu:G-quadratic}(i). The reader interested in the concrete examples may skip to \Cref{prop:smallest counterexample} for an Artinian Gorenstein example in four variables and to \Cref{prop:many counterexamples} for an infinite family of Artinian level examples with embedding dimension greater than or equal to five.

For starters, let us note that it is easy to find examples of algebras that are strongly Koszul with respect to variables, but do not admit a Gr\"obner basis of quadrics in the given coordinates: for instance, one checks that the ideal $I = (x^2+y^2, xy) \subseteq \field[x,y]$ exhibits such a behavior. However, the linear change of coordinates sending $x$ into $x+y$ transforms $I$ into the ideal $(x^2, xy+y^2)$, which does admit a Gr\"obner basis of quadrics. To get a genuine example of a strongly Koszul algebra that does not admit a Gr\"obner basis of quadrics even after a change of coordinates, we need to try a little harder.

Eisenbud, Reeves and Totaro found in \cite[Theorem 19]{ERT} an obstruction for an ideal to have a quadratic initial ideal. We state below a corollary of their result that will be enough for our purposes.

\begin{lemma}[compare {\cite[Theorem 19]{ERT}}] \label{lem:obstruction to G-quadraticity}
Let $S = \field[x_1, \ldots, x_n]$ and let $I \subseteq S$ be a quadratic homogeneous ideal such that $S/I$ is Artinian. If there exists no linear form $\ell \in S_1$ such that $\ell^2 \in I$, then $I$ does not admit any Gr\"obner basis of quadrics, even after a linear change of coordinates.
\end{lemma}

Provided the characteristic of $\field$ is different from $2$, the existence of linear forms whose square lies in the quadratic ideal $I$ can be investigated using the differentiation pairing.

\begin{definition} \label{def:orthogonality}
    Let $d \in \mathbb{N}$, $\field$ be a field of characteristic $0$ or greater than $d$, and $S = \field[x_1, \ldots, x_n]$. Then the symmetric bilinear pairing $S_d \times S_d \to \field$ obtained by partial differentiation as in \Cref{subsec:differentiation} is nonsingular. (Note that in this setting the characteristic of $\field$ is allowed to be positive, provided it is greater than $d$.) In what follows, we will use $(-)^{\perp}$ to denote orthogonality with respect to this pairing. 
\end{definition}

Note that the claim of the next lemma would \emph{not} hold if we used the contraction pairing instead of the differentiation one.

\begin{lemma}[compare {\cite[Lemma 2.12]{DKoszul}}] \label{lem:apolarity}
Let $\field$ be a field of characteristic $0$ or greater than $d$, let $S = \field[x_1, \ldots, x_n]$ and let $V$ be a vector subspace of $S_d$. Then there exists a bijection between the nonzero linear forms $\ell \in S_1$ such that $\ell^d \in V$ and the points of the projective variety whose defining equations are a $\field$-basis of $V^{\perp}$.
\end{lemma}

If we further assume that $\field$ is algebraically closed, then we can use the weak projective Nullstellensatz \cite[Theorem 8.3.8]{CLO} to derive the following corollary of \Cref{lem:apolarity}:

\begin{corollary} \label{cor:no d-th powers}
Let $\field$ be an algebraically closed field of characteristic $0$ or greater than $d$, let $S = \field[x_1, \ldots, x_n]$, let $V$ be a vector subspace of $S_d$ and let $I$, $J$ be the ideals of $S$ generated by $V$ and $V^{\perp}$, respectively. Then there exist no nonzero linear forms $\ell \in S_1$ such that $\ell^d \in I$ if and only if the quotient ring $S/J$ is Artinian.
\end{corollary}

\begin{remark} \label{rem:arithmetic}
    Let $\field$ be a field of characteristic different from $2$. When $\field$ is not quadratically closed, i.e., when there exists at least an element $a \in \field$ such that the equation $z^2=a$ has no solutions in $\field$, one can modify slightly the example in the discussion after \Cref{qu:G-quadratic} to get the ideal $I = (x^2+ay^2, xy) \subseteq \field[x,y]$. Then $\{x^2-ay^2\}$ is a $\field$-basis for $(I_2)^{\perp}$ and hence, using \Cref{lem:apolarity}, there are no nonzero linear forms in $S$ whose square lies in $I$. By \Cref{lem:obstruction to G-quadraticity}, the ideal $I$ does not admit any quadratic Gr\"obner basis, even after a change of coordinates.
\end{remark}    

In order to be able to apply \Cref{cor:no d-th powers} and overcome arithmetic obstacles like the ones in \Cref{rem:arithmetic}, for the rest of the section we will often assume that the field $\field$ is algebraically closed. Moreover, in order to use the differentiation pairing in \Cref{def:orthogonality}, we will further assume that the characteristic of $\field$ is different from $2$. 

We are now ready to show that \Cref{qu:G-quadratic}(i) has a negative answer. We first present an Artinian Gorenstein example in four variables. When working over the complex numbers, such a ring corresponds to the  defining equation of the Clebsch cubic surface via Macaulay's inverse system.

\begin{proposition} \label{prop:smallest counterexample}
    Let $\field$ be any field, $S = \field[x,y,z,t]$, $I = (x^2-yz, y^2-zt, z^2-tx, t^2-xy, xz, yt)$ and $R = S/I$. Then:
    \begin{enumerate}
    \item the ring $R$ is Artinian Gorenstein, has Hilbert function $(1,4,4,1)$ and corresponds via Macaulay's inverse system (using the contraction action) to the cubic $X^2Y+Y^2Z+Z^2T+T^2X \in \mathcal{D}$;
    \item the ideal $I$ admits a tidy universal Gr\"obner basis of quadrics and cubics, namely,
    \[\begin{split}\mathcal{G} = \{&x^2-yz, y^2-zt, z^2-tx, t^2-xy, xz, yt,\\&x^3, y^3, z^3, t^3, xy^2, yz^2, zt^2, tx^2, x^2y-z^2t, y^2z-t^2x\};\end{split}\]
    \item the ring $R$ is strongly Koszul with respect to the variables;
    \item if moreover $\field$ is an algebraically closed field of characteristic different from $2$, $3$ and $5$, then the ideal $I$ has no quadratic Gr\"obner basis, even after a linear change of coordinates.
    \end{enumerate}
\end{proposition}
\begin{proof} 
Note that the ideal $I$ is generated by monomials and differences of monomials; as a consequence, all Gr\"obner-based computations will yield the same results regardless of the characteristic of the field $\field$.

 For parts (ii) and (iii) it is enough to fix a specific field (say, the rationals $\mathbb{Q}$) and check the computations by hand or with the aid of a computer algebra system like \texttt{Macaulay2} \cite{M2}. One can compute $\mathcal{G}$ from scratch using \texttt{Gfan} \cite{gfan} or its Macaulay2 port \texttt{gfanInterface} \cite{gfanInterface}, or simply check that all possible S-pairs of elements in $\mathcal{G}$ reduce to zero.

 For part (i), let us consider $F := X^2Y+Y^2Z+Z^2T+T^2X \in \mathcal{D}$. Since $F$ is a cubic and the $\field$-vector space generated by $x \circ_{\mathrm{ctr}} F, \ y \circ_{\mathrm{ctr}} F, \ z \circ_{\mathrm{ctr}} F, \ t \circ_{\mathrm{ctr}} F$ is $4$-dimensional, one has that the Hilbert function of the Artinian Gorenstein ring $S/\mathrm{ann}(F)$ must be $(1,4,4,1)$. One checks immediately that all the generators of $I$ lie in $\mathrm{ann}(F)$ as well. Since $S/I$ also has Hilbert series $(1,4,4,1)$, as can be checked via characteristic-free Gr\"obner methods, it must then be that $I = \mathrm{ann}(F)$.  

 As for part (iv), let us now assume that $\field$ is an algebraically closed field of characteristic different from $2$, $3$ and $5$. By \Cref{lem:obstruction to G-quadraticity}, it is enough to prove that there exists no linear form $\ell \in S_1$ such that $\ell^2 \in I$. Since $\field$ is algebraically closed and $\mathrm{char}(\field) \neq 2$ by assumption, the hypotheses of \Cref{cor:no d-th powers} are met. Let $J$ be the ideal of $S$ generated by the vector space $(I_2)^{\perp}$, where $(-)^{\perp}$ is as in \Cref{def:orthogonality}. Since $\dim_{\field}((S/I)_2) = 4$, we have that $\dim_{\field}((I_2)^{\perp}) = 4$ as well. A $\field$-basis for $(I_2)^{\perp}$, and thus a minimal generating set for the ideal $J$, is given by $\{x^2+2yz, y^2+2zt, z^2+2tx, t^2+2xy\}$. If we can prove that $S/J$ is Artinian (actually, an Artinian complete intersection of quadrics), we are done by \Cref{cor:no d-th powers}. In $S/J$ one has that $\overline{xt^2} = -2\overline{x^2y} = 4\overline{y^2z} = -8\overline{z^2t} = 16\overline{xt^2},$ whence $15\overline{xt^2} = \overline{0}$. Since $\mathrm{char}(\field) \notin \{3, 5\}$, it follows that $\overline{xt^2} = \overline{0}$, and analogously $\overline{x^2y} = \overline{y^2z} = \overline{z^2t} = \overline{0}$. Then $\overline{x^4} = -2\overline{x^2yz} = \overline{0}$, and analogously $\overline{y^4} = \overline{z^4} = \overline{t^4} = \overline{0}$, which proves that $S/J$ is Artinian. This concludes the proof.
\end{proof}

\begin{remark}
Let $\field$ be an algebraically closed field of characteristic zero. Under this assumption, another way to prove part (iv) of \Cref{prop:smallest counterexample} can be found in the paper \cite{ConcaRossiValla} by Conca, Rossi and Valla. Note first that, using the differentiation pairing instead of the contraction one, it holds that $I = \mathrm{ann}(F)$, where $F := X^2Y+Y^2Z+Z^2T+T^2X \in \mathcal{S}$. Since the cubic hypersurface $X^2Y+Y^2Z+Z^2T+T^2X$ is smooth (e.g., by the Jacobian criterion), applying \cite[Proposition 6.2]{ConcaRossiValla} yields that the ideal $I$ does not admit a Gr\"obner basis of quadrics even after a linear change of coordinates.
\end{remark}

\begin{remark}
It can be proved analogously to \Cref{prop:smallest counterexample} that, if $I$ is the ideal apolar (via the contraction pairing) to the polynomial $X_1^2X_2 + X_2^2X_3 + X_3^2X_4 + X_4^2X_5 + X_5^2X_1 \in \mathcal{D}$, then $S/I$ is an Artinian Gorenstein ring that is strongly Koszul with respect to the variables. If moreover $\field$ is algebraically closed of characteristic different from $2$, $3$ and $11$, then $I$ has no quadratic Gr\"obner basis, even after a linear change of coordinates. However, trying to extend this example to six variables (i.e., considering the ideal apolar to $X_1^2X_2 + X_2^2X_3 + X_3^2X_4 + X_4^2X_5 + X_5^2X_6 + X_6^2X_1$) fails, as the resulting ring is not strongly Koszul with respect to the variables.
\end{remark}

We conclude this paper by constructing an infinite family of ideals exhibiting a behavior analogous to the ideal $I$ from \Cref{prop:smallest counterexample}.

For any $n \geq 3$, let $C_n$ be the cycle graph on $n$ vertices, i.e.,~the finite simple graph with edges $\{1,n\}$ and $\{i,i+1\}$ for every $i \in \{1,2,\ldots,n-1\}$.

\begin{proposition} \label{prop:many counterexamples}
Let $\field$ be any field, $n \geq 5$, $S = \field[x_1, \ldots, x_n]$, $k = \lfloor\frac{n+1}{2}\rfloor$ and \[I = (x_i^2-x_{i+k-1}x_{i+k} \mid 1 \leq i \leq n) + (x_ix_j \mid i\neq j, \ \{i,j\} \notin E(C_n)),\]
where by convention indices are taken modulo $n$: e.g.,~$x_{n+j} = x_j$. Then:
\begin{enumerate}
\item the ring $S/I$ is Artinian with Hilbert function $(1, n, n)$; \item the ring $S/I$ is level and the ideal $I$ is apolar (via the contraction action) to the $S$-submodule $M$ of $\mathcal{D}$ generated by $F_1, \ldots, F_n$, where $F_i := X_i^2+X_{i+k-1}X_{i+k}$ and indices are taken modulo $n$;
\item the ideal $I$ admits a tidy universal Gr\"obner basis consisting of $\binom{n}{2}$ quadrics and $3n$ cubics, namely,
\[\begin{split}\mathcal{G} = \{x_i^2-x_{i+k-1}&x_{i+k} \mid 1 \leq i \leq n\} \cup \{x_ix_j \mid i\neq j, \ \{i,j\} \notin E(C_n)\}\\&\cup \{x_i^3, x_i^2x_{i+1}, x_i^2x_{i-1} \mid 1 \leq i \leq n\}\end{split}\]
(where all indices are again taken modulo $n$);
\item the ring $S/I$ is strongly Koszul with respect to the variables;
\item if moreover $\field$ is an algebraically closed field of characteristic different from $2$ and from any prime divisor of $2^n + (-1)^{n+1}$, then the ideal $I$ has no quadratic Gr\"obner basis, even after a linear change of coordinates.
\end{enumerate}
\end{proposition}

\begin{proof}
Note that, just like in \Cref{prop:smallest counterexample}, the ideal $I$ is generated by monomials and differences of monomials, and hence all Gr\"obner-based computations will yield the same results regardless of the characteristic of the field $\field$.
\begin{enumerate}
\item The ideal $I$ is generated by $n + \frac{n(n-3)}{2} = \binom{n}{2}$ linearly independent quadrics, whence $\dim_{\field}(S/I)_2 = \binom{n+1}{2} - \binom{n}{2} = n$. We now want to show that every monomial of degree $3$ belongs to $I$. Every squarefree cubic monomial contains at least two variables that are not adjacent in the $n$-cycle $C_n$; hence, such a monomial belongs to $I$. Due to the cyclic $\mathbb{Z}_n$-symmetry of the generators of $I$, it is now enough to check that $x_1^3$, $x_1^2x_2$ and $x_1x_2^2$ belong to $I$. Working in the quotient ring $S/I$, one has that $\overline{x_1^3} = \overline{x_1x_kx_{k+1}} = \overline{0}$, $\overline{x_1^2x_2} = \overline{x_2x_kx_{k+1}} = \overline{0}$ and $\overline{x_1x_2^2} = \overline{x_1x_{k+1}x_{k+2}} = \overline{0}$ since $2 < k < n-1$. 
\item Since the $\field$-vector space generated by $\{x_i \circ_{\mathrm{ctr}} F_j \mid i,j \in \{1, \ldots, n\}\}$ is $n$-dimensional, one has that the Hilbert function of $S/\mathrm{ann}(M)$ is $(1,n,n)$. One checks immediately that all generators of $I$ belong to $\mathrm{ann}(M)$ and thus, since the Hilbert functions coincide, $I = \mathrm{ann}(M)$.
\item Since $(x_1, \ldots, x_n)^3 \subseteq I$ by part (i), it follows that the generators of $I$ and all the cubic monomials of $S$ form a (tidy) universal Gr\"obner basis for $I$; one can then throw away the cubic monomials that are divisible by $x_ix_j$ for some distinct $i$ and $j$ not adjacent in the $n$-cycle $C_n$, obtaining the set $\mathcal{G}$.
\item Let us pick a proper subset $Y$ of $\{x_1, \ldots, x_n\}$ and $x \in \{x_1, \ldots, x_n\} \setminus Y$. We need to ensure that the colon ideal $\overline{Y} :_R \overline{x}$ is generated by variables, where $R = S/I$. By part (iii) and \Cref{cor:tidy revlex UGB}, we know beforehand that $\overline{Y} :_{R} \overline{x}$ is generated by monomials in $R$. Due to the cyclic symmetry of the generators of $I$, we can assume without loss of generality that $x = x_1$. It is immediate that $(\overline{x_3}, \overline{x_4}, \ldots, \overline{x_{n-1}}) \subseteq \overline{Y} :_R \overline{x_1}$. If we can prove that all quadratic monomials in the variables $\overline{x_1}, \overline{x_2}, \overline{x_n}$ belong to $(\overline{x_3}, \overline{x_4}, \ldots, \overline{x_{n-1}})$, we are done; indeed, if that is the case, the ideal $\overline{Y} :_R \overline{x_1}$ is generated by $\overline{x_3}, \overline{x_4}, \ldots, \overline{x_{n-1}}$ and possibly some of $\overline{x_1}, \overline{x_2}, \overline{x_n}$.

We need to treat separately the cases when $n$ is odd (recalling that $k = \lfloor\frac{n+1}{2}\rfloor$, this implies that $n=2k-1$) and when $n$ is even (which implies that $n=2k$). Note that, in both cases, the monomial $x_2x_n$ belongs to $I$, and thus $\overline{x_2x_n} = \overline{0}$.

\textbf{Case $n = 2k-1$:} since $k \geq 3$, it holds that $k+1 < 2k-1 = n$. In particular, $x_k$ and $x_{k+1}$ both belong to $\{x_3, x_4, \ldots, x_{n-1}\}$. Then, since $\overline{x_1}^2 = \overline{x_kx_{k+1}}$, $\overline{x_1x_2} = \overline{x_{k+1}}^2$, $\overline{x_1x_n} = \overline{x_k}^2$, $\overline{x_2}^2 = \overline{x_{k+1}x_{k+2}}$ and $\overline{x_n}^2 = \overline{x_{k-1}x_k}$, we are done.

\textbf{Case $n = 2k$:} since $k \geq 3$, it holds that $k+2 < 2k = n$. In particular, $x_k$, $x_{k+1}$ and $x_{k+2}$ all belong to $\{x_3, x_4, \ldots, x_{n-1}\}$. Then, since $\overline{x_1}^2 = \overline{x_kx_{k+1}}$, $\overline{x_1x_2} = \overline{x_{k+2}}^2$, $\overline{x_1x_n} = \overline{x_{k+1}}^2$, $\overline{x_2}^2 = \overline{x_{k+1}x_{k+2}}$ and $\overline{x_n}^2 = \overline{x_{k-1}x_k}$, we are done.

\item As in the proof of \Cref{prop:smallest counterexample}, since $\field$ is an algebraically closed field of characteristic different from $2$, it is enough to prove that the ideal $J$ of $S$ generated by the vector space $(I_2)^{\perp}$ gives rise to an Artinian quotient ring (actually, a complete intersection of quadrics). Since $\dim_{\field}((S/I)_2) = n$, we have that $\dim_{\field}((I_2)^{\perp}) = n$ as well. A $\field$-basis for $(I_2)^{\perp}$, and thus a minimal generating set for the ideal $J$, is given by $\{x_i^2+2x_{i+k-1}x_{i+k} \mid 1 \leq i \leq n\}$ (where indices are taken modulo $n$). By the cyclic symmetry of the generators of $J$, it is enough to prove that $x_1^M \in J$ for some $M$ large enough. In order to get there, it is enough to show that the following two claims hold:

\textbf{Claim 1:} $J$ contains the monomial $x_1^{2}x_2^{2}\ldots x_n^{2}$.

\textbf{Claim 2:} for $M$ large enough, $\overline{x_1^M}$ is a multiple of $\overline{x_1^{2}x_2^{2}\ldots x_n^{2}} = \overline{0}$ in $S/J$. 

Let $c \cdot {\mathbf{x}}^{\mathbf{b}} = c \cdot x_1^{b_1}x_2^{b_2}\ldots x_n^{b_n} \in S$. Note that, if $b_i \geq 2$, then $c \cdot \overline{\mathbf{x}}^{\mathbf{b}} = -2c \cdot \overline{\mathbf{x}}^{\mathbf{b'}}$ in $S/J$, where $b'_i = b_i - 2$, $b'_{i+k-1} = b_{i+k-1}+1$, $b'_{i+k} = b_{i+k}+1$ and $b'_j = b_j$ for all other indices $j$. In what follows, provided that $b_i \geq 2$, we will say that the \emph{$i$-th move} transforms $c \cdot {\mathbf{x}}^{\mathbf{b}}$ into $-2c \cdot {\mathbf{x}}^{\mathbf{b'}}$. If one can perform in a sequence all $n$ possible moves on a certain ${\mathbf{x}}^{\mathbf{b}}$, it follows that $\overline{\mathbf{x}}^{\mathbf{b}} = (-2)^n \cdot \overline{\mathbf{x}}^{\mathbf{b}}$, and thus $(1 - (-2)^{n}) \cdot {\mathbf{x}}^{\mathbf{b}} \in J$. Due to the hypothesis on the characteristic of $\field$, one then has that $(1 - (-2)^{n}) \neq 0$, and thus ${\mathbf{x}}^{\mathbf{b}} \in J$.

To prove Claim 1 it now suffices to observe that one can perform all $n$ possible moves in a sequence on $x_1^{2}x_2^{2}\ldots x_n^{2}$, and thus $x_1^{2}x_2^{2}\ldots x_n^{2} \in J$.

To prove Claim 2, observe first that $\overline{x_1^4}$ is a multiple of $\overline{x_n}$; indeed, applying the first move twice and then the $k$-th move one gets that $\overline{x_1^4} = -8\overline{x_{k+1}^2 x_{2k-1}x_{2k}}$, and the latter monomial contains both $x_{2k-1}$ and $x_{2k}$ (one of which is $x_n$, depending on the parity). Due to the cyclic symmetry on the possible moves, one has that $\overline{x_n^4}$ is a multiple of $\overline{x_{n-1}}$, and so on. It follows that, for $M$ large enough, the monomial $\overline{x_1^M}$ equals a multiple of $\overline{x_1^{2}x_2^{2}\ldots x_n^{2}} = \overline{0}$ in $S/J$, which ends the proof.
\end{enumerate}
\end{proof}

\bibliographystyle{alpha} 
\bibliography{SK_bibliography}
\end{document}